\documentclass[article]{svjour3}
\smartqed
\usepackage[margin=1in]{geometry}
\usepackage{graphicx}
\usepackage{setspace}
\usepackage{amsmath,amssymb}

\newcommand{\R}{\mathbb{R}}

\newcommand{\N}{\mathbb{N}}

\newcommand{\hsm}{\hspace{-.1cm}}
\newcommand{\hsmm}{\hspace{-.05cm}}
\newcommand{\vsm}{\vspace{-.1cm}}

\newcommand{\dis}{\displaystyle}
\newcommand{\noi}{\noindent}

\begin{document}

\title{A Maximum Principle for a Time-Optimal Bilevel Sweeping Control Problem}

\subtitle{}

\author{Fernando Lobo Pereira \and Nathalie T. Khalil}

\institute{Fernando Lobo Pereira,  Corresponding author  \at
              SYSTEC, Department of Electrical and Computer Engineering, Porto University \\
              4200-465 Porto, Portugal\\
              flp@fe.up.pt
           \and
              Nathalie T. Khalil \at
              SYSTEC, Department of Electrical and Computer Engineering, Porto University \\
              4200-465 Porto, Portugal\\
              khalil.t.nathalie@gmail.com
}

\date{Received: date / Accepted: date}

\maketitle

\begin{abstract}
 In this article, we investigate a time-optimal state-constrained bilevel optimal control problem whose lower-level dynamics feature a sweeping control process involving a {\it truncated normal cone}. By bilevel, it is meant that the optimization of the upper level problem is carried out over the solution set of the lower level problem.This problem instance arises in structured crowd motion control problems in a confined space.
 We establish the corresponding necessary optimality conditions in the Gamkrelidze's form. The analysis relies on the {\it smooth approximation} of the lower level sweeping control system, thereby dealing with the resulting lack of Lipschitzianity with respect to the state variable inherent to the sweeping process, and on the {\it flattening} of the bilevel structure via an exact penalization technique. Necessary conditions of
 optimality in the Gamkrelidze's form are applied to the resulting standard approximating penalized state-constrained single-level problem, and the main result of this article is obtained by passing to the limit.
\end{abstract}

\section{Introduction}\label{sec: introduction}
In this article, we present and derive a Maximum Principle in which the usual Pontryagin-Hamilton function (cf. \cite{dubovitskii1965extremum}) is replaced by its Gamkrelidze's form (cf. \cite{gamkrelidze1959,arutyunov2011maximum}) for a certain state constrained bilevel optimal control problem featuring a sweeping process at the lower level.
The first formulation of an instance of this problem is the minimum-time optimal control problem of a moving structured crowd in \cite{khalil2019bilevelsweeping} in which existence, and well-posedness are addressed. To the best of our knowledge this is the first time that a maximum principle is derived for this class of problems. The motivation to investigate this class of problems, goes well beyond motion control of structured crowds. Multi-level sweeping control process paradigms constitute a natural framework for the multiple controlled dynamic systems like, for example, operation of teams of drones or robots, in a shared confined space.

Sweeping processes were introduced in the 1970's by Moreau in \cite{moreau1976application} to model elastoplasticity phenomena which involved a differential inclusion of the form\vsm\vsm\vsm
\begin{equation} \dot x(t)\in -N_{C}(x(t))\;\;\mbox{ a.e.}\;t\in[0,T],\label{swp}\vsm\vsm\vsm\end{equation}
where $C$ is a convex set, and $N_{\Omega}(x)$ is the normal cone to a convex set $\Omega\subset \R^n$ at $x$. The sequel of his research led to the so-called {\it catching-up algorithm}, \cite{moreau1999numerical}. Fuelled by a fast expansion of the range of applications, which encompasses systems with hysteresis, nonlinear electric circuits, nonsmooth mechanics, robotics, traffic equilibria, and crowd motion, to name just a few, see \cite{maury2008mathematical,venel2011numerical,castaing2014some,Brogliato2018higher}, this field expanded to very diverse sweeping processes settings. This has been documented in a vast literature on a rather comprehensive body of results in control and optimization encompassing time-varying, \cite{arroud2018maximum,brokate2013optimal}, nonconvex sweeping sets, \cite{Thibault2003,Thibault2008}, dynamics with additive nonlinear controlled drifts, \cite{edmond2005relaxation,colombo2012optimal,colombo2016minimum,de2019optimal,zeidan2020sweeping,colombo2020optimization}, distributed systems, \cite{adly2014evolution}, and systems with trajectories of bounded variation, among others, being these references a very small sample. Moreover, it is worth pointing out that, by considering discrete approximations, the relevance of the developments in \cite{cao2016optimal,hoang2019Euler,cao2019optimal} goes beyond contributions in dynamic optimization theory to include also computational schemes.

In order to provide a flavor of the bilevel optimal control paradigm involving sweeping processes, consider a population confined to a certain closed space - assumed to be a sphere, for simplicity - and trying to exit it in the minimal possible time. The population is moving along a trajectory $y$ prescribed by a coordinator in order to reach the exit. Moreover, the population features a certain {\it structure}, in the sense that it is constituted of groups of individuals each one remaining in a moving set (also a closed sphere), while minimizing their effort to achieve this. In order to model the ``conflict'' between the goal of minimizing the time to reach the exit, and the one of minimizing the effort in order to stay inside the set moving in the direction of the exit, the problem is formulated as a bilevel problem as follows:\vsm
\begin{itemize}
\item Upper level control problem defining the motion direction for each group with the goal of driving the whole population to the exit in minimum time, while all groups remain constrained to the big closed sphere, and
\item Lower level control problem whose dynamics define the motion of the individuals of each group modelled by a sweeping control process with its own control dependent drift term required to stay confined to their moving set while minimizing their control effort to achieve this.
\end{itemize}
Given the complexity of the optimality conditions for the general formulation of this problem, here, we consider only one population group.
Thus, by considering two closed spheres $Q$, and $ Q_1$ in $\R^n$ with radius $R$, and $R_1$, respectively, with $R \gg R_1$, and centered,
respectively, $q_0$, and $y_0$ in $ \R^n$ such that $y_0+ Q_1\subset Q$, our set-up involves two nested optimal control problems:
\begin{itemize}
\item Upper level minimum-time control problem $(P_H(x(0),u_L))$ whose control process $(T,y)$ is such that $y(0) = y_0$ is a given point in $ Q$
    satisfying $Q_1 + y(t) \subset  Q$ $\forall\; t\in [0,T]$ and $y(T) \in \bar E$, being $\bar E$ the exit set to be defined below, $T$ the final time, and $y$ a vector defining the translation of the disk $Q_1$ at each time, while $(x(0),u_L)$ is a solution to the lower level problem.
\item Lower level minimum control effort problem $(P_L(T,y))$ whose control process $(x,u_L)$, is such that $x(0)\in Q_1 + y_0$, and $x(t) \in Q_1 + y(t) $ $\forall\; t\in [0,T]$, where $x \in \R^n$ is the ``representative'' position of the individuals of the group. $(T,y)$ is regarded as a parameter inherited from the upper level problem.
\end{itemize}
Now, we state both optimal control problems in detail.\vsm\vsm\vsm
\begin{eqnarray}(P_H(x(0),u_L)) &\mbox{\hspace{.2cm}}& \mbox{Minimize }  J_H(T,y;x(0),u_L), \nonumber\label{uppercost}\\
\mbox{subject to }&& \dot y(t)= v(t)\quad \mbox{a.e. in }[0,T],\quad  y(0)= y_0 \in Q, \nonumber\\
&& y(T)\in \bar E,\;\bar E=\partial[(E+R_1B_1(0))\cap Q],\; E \subset \partial Q,\label{upperendp}\nonumber\\
&&  v \in \mathcal{V}:= \{v\in L_2([0, T]; \R^n): \, v(t)\in V\},\nonumber\\
&&  Q_1 + y(t) \subset  Q  \quad \forall\,t\in [0,T], \quad \mbox{and} \label{uppersetconst-1}\\
&& (T,y) \mbox{ is such that } \exists \mbox{ a solution } (x(0),u_L) \textrm{ to } (P_L(T,y)).\vsm\vsm\label{lower-level feasibility}
\end{eqnarray}

Here, $y_0$ is the given center of $Q_1$ at the initial time $t=0$, $J_H(T,y;x(0),u_L):= T$ (time-optimal), $V \subset\R^n $ is compact, $E$ is closed and connected, $B_1(0)$ is the closed unit ball in $\R^n$ centered at the origin, $\partial A$ denotes the boundary of the set $A$, and $(P_L(T,y))$ represents the parametric lower level control problem whose dynamics involve a sweeping process. Before pursuing, we remark that $\bar E$ was designed as target set for $y$ to ensure a feasible finite time solution preserving (\ref{uppersetconst-1}).\vsm\vsm
\begin{eqnarray}
(P_L(T,y)) \nonumber &\mbox{\hspace{.2cm}}&\mbox{Minimize } J_L(x(0),u_L;T,y), \label{cost-i}\nonumber\\
\mbox{subject to }&& \dot x(t)\in f(x(t),u(t))- N_{Q_1+y(t)}^{M} (x(t)), \quad x(0) \in Q_1+y_0 \label{sweep-i}\mbox{ a.e. } \\
&& u\in\mathcal{U} := \{ u\in L_\infty([0,T];\R^m) : \, u(t)\in U\},\nonumber\\
&& x(t)\in Q_1+y(t)\quad \forall t\in [0,T],\label{concons-i}\vsm\vsm
\end{eqnarray}

\vspace{-.4cm}
\noindent where, for a given pair $(T,y)$, $\dis J_L(x(0),u_L;T,y)\hsmm :=\hsm \int_{0}^{T}\hsmm\hsm\hsm \left( |u(t)|^2\hsm +\hsmm  |u_0(t)|^2 \right) dt$\footnote{Here, $|u(t)|$, and $|u_0(t)|$ are the finite dimensional norms of the values of $u$ and $u_0$, respectively, as functions in $L_\infty$} (control effort), $f\hsmm :\R^n\hsmm\times\hsmm \R^m \hsmm\to\hsmm\R^n$, $U\hsmm \subset\hsmm \R^m$ is compact, the truncated cone $ N_{Q_1+y(t)}^{M}(z)\hsmm :=\hsmm N_{Q_1+y(t)}(z) \cap  M B_1(0)$, being $N_{Q_1+y(t)}(z)$ the Mordukhovich (limiting) normal cone to the closed set $Q_1+y(t)$ at point $z$ in the sense of \cite{Mordukhovich2006,clarke1990optimization}, with $M\hsmm >\hsmm 0$ being a given constant, and $ u_0\hsmm\in\hsmm\mathcal{U}_0\hsmm :=\hsmm  L_\infty([0,T];[0,1])$ is a control function taking nonzero values only on the set $\{t\hsmm\in\hsmm 0,T]: x(t)\hsmm\in\hsmm\partial [Q_1\hsmm +\hsmm y(t)]\}$, and specified implicitly in the sweeping component of the dynamics by

\vspace{-.6cm}
\begin{equation} - N_{Q_1+y(t)}^{M} (x(t)):=\{\bar f (x,y)(x-y)u_0(t): u_0(t)\in[0,1]\},\label{alternate-Normalcone}\end{equation}

\vspace{-.1cm}
\noindent being $\dis \bar f (x,y) = -\frac{M}{R_1}$ if $x\in\partial [Q_1+y] $, and $ 0 $ otherwise.

Remark that, while $(P_L(T,y))$ is a nonstandard sweeping optimal control problem that depends on the parameter $(T, y)$, $(P_H(x(0),u_L))$, in turn, not only depends on the parameter $(x(0),u_L)$, but also features the constraint (\ref{lower-level feasibility}) whose satisfaction entails the nonemptiness of the solution set of $(P_L(T,y))$.
\begin{remark}
The reason to consider the truncated normal cone at the lower level sweeping process instead of the usual normal cone is to preserve the relevance
of the bilevel structure of the problem. If the normal cone were used in the problem formulation, then the upper level control could always drive arbitrarily the set $Q_1$ to the exit set $E$ without forcing the lower level control system to use an extra control effort to remain feasible, and the bilevel problem would collapse to two independent optimal control problems. Results reported in \cite{Thibault2003,Thibault2008} prove that, under certain conditions, the solutions to a sweeping process formulated with the usual normal cone to a continuously time dependent set, $C(t)$, coincide with those whose normal cone is replaced by $g(t)\partial d_{C(t)}(x)$ for a certain function $g(\cdot)$ reflecting the magnitude of the variation of $C$ with time. Thus, no time independent truncated normal cones are considered.
\end{remark}
Whenever the context is clear, we refer to the upper and the lower level problems as $(P_H)$, and $(P_L)$, respectively, thus, simplifying the presentation by dropping the parameters $(x(0),u_L)$ and $(y,T)$.

For some parameter $(T,y)$, a pair $(x(\cdot), u_L(\cdot))$ is a {\it feasible} (or admissible) control process to $(P_L)$ if $u_L(\cdot)$ is feasible control to $(P_L)$, and $x(\cdot)$ is an arc satisfying the differential inclusion (\ref{sweep-i}), the initial condition, together with (\ref{concons-i}). An {\it optimal solution} to $(P_L)$ is a feasible pair of $(P_L)$ minimizing the value of the cost functional $J_L(x(0),u_L;T,y)$ over all admissible pairs of $(P_L)$. A feasible quadruple of the dynamic control problem $(P_H)$ is the collection of a feasible pair $(T,y)$, and an optimal pair $(x(0),u_L)$ to $(P_L)$. The feasible quadruple $(T,y;x(0),u_L)$ is optimal to $(P_H)$ if $(T,y;x(0),u_L)$ if it minimizes the value of $J_H(T,y;x(0),u_L)$ among all admissible strategies of $(P_H)$.
Given our set-up, the state constraints (\ref{uppersetconst-1}), and (\ref{concons-i}) can be respectively expressed by inequality constraints $h_H(y(t))\hsm\leq\hsm 0 $, and $ h_L(x(t),y(t)) \hsm\leq\hsm  0 $, where, by denoting  the Euclidean norm by $|\cdot|$,\vsm\vsm\vsm
\begin{eqnarray} \label{eq:expression of h_H}h_H(y)&:= &\frac{1}{2}\left( |y-q_0|^2-(R-R_1)^2\right), \mbox{ and}\vsm\vsm \\
\label{eq:expression of h_L} h_L(x,y)&:=& \frac{1}{2}\left( |x- y|^2-R_1^2\right). \end{eqnarray}
We emphasize that the presence of state constraints is the reason why we chose the Gamkrelidze form of the Maximum Principle, \cite{gamkrelidze1959,arutyunov2011maximum} instead of the more common Dubovitskii-Myliutin form, \cite{dubovitskii1965extremum}. While, in the latter, the multipliers associated with the state constraints are measures, in the former, they are the much more regular monotonic functions of bounded variation. The price to pay for this extra regularity is that the functions specifying the inequality state constraints have to be $C^2$. This assumption clearly holds for our problem. Moreover, as shown in \cite{AKP2017,KP2019}, if the problem's data satisfies additional controllability, and regularity assumptions, then these multipliers turn out to be continuous. This property is critical to ensure the efficiency of numerical algorithms based on indirect methods using necessary conditions of optimality, a fact that was exploited in \cite{CKKP2021,PCDDKS2021}.

The article is organized as follows. In section \ref{sec: main theorem}, the assumptions on the data of the problem, and the main theorem - necessary optimality conditions in the Gamkrelidze's form of the time-optimal bilevel sweeping control problem defined by $(P_H)$ and $(P_L)$ - are stated.
In Section \ref{sec:proof}, the proof of the main theorem is presented in detail with proofs of auxiliary results left in the appendix. The last section, \ref{sec: conclusion}, provides some conclusion and outlines future avenues of this work.

{\bf Notation.} $\partial^P \varphi$, $\partial \varphi$, and $\partial^C \varphi$ denote, respectively, the proximal, the Mordukhovich (limiting), and the Clarke subdifferentials of the function $\varphi$. If $\varphi$ is locally Lipschitz, then $\partial^C \varphi= \textrm{co} \partial \varphi$, where ``$\textrm{co} A$'' denotes the closure of the convex hull of the set $A$. Nonsmooth analysis concepts, results and tools can be consulted in \cite{Mordukhovich2006,clarke1990optimization,vinter2010optimal}. $AC([0,T];\R^n)$ stands for the set of $\R^n$ valued absolutely continuous functions on $[0,T]$, while $N\hspace{-.03cm}BV([0,T];\R)$ is the set of nonnegative valued scalar functions of bounded variations on $[0,T]$, and $\|\cdot\|_{TV}$ is the total variation.

\section{Main Theorem}\label{sec: main theorem}
Our main result - necessary conditions of optimality in Gamkrelidze's form for the bilevel optimal control problem formulated in the previous
section - of this article is stated in this section. We start by stating and discussing the presenting assumptions $H1-H6$, and, then, move to the partial calmness assumption required to the application of the exact penalization technique used in the proof.

The data of $(P_H)$ and of $(P_L)$ satisfies the following standing assumptions:

\vsm\vsm\vsm
\begin{itemize}
\item[H1] $f(x,\cdot)$ is Borel measurable $\forall\,x\in \R^n$, $f(\cdot,u)$ is Lipschitz continuous with constant $K_f$ for all $u\in U$, and there exists a constant $M_1>0$ such that, for all $(x,u) \in \R^n\hsm \times\hsm U$, $|f(x,u)| \le M_1$.
\item[H2] $f(x,U)\subset\R^n$ is a closed and convex set for each $x$.
\item[H3] The control constraint sets $U$, and $V$ are compact and convex.
\item[H4] There exists $\delta > 0$ such that $\delta B_1(0) \subset f(x,U),\, \forall\, x\in \R^n$.
\item[H5] The constant $M > 0$ specifying the truncation level of the normal cone has to satisfy $\overline M> M > \overline m$ where, $ \forall \, \zeta\in  N_{Q_1+y(t)}(x(t))$, $\forall \, t\in  [0,T]$ such that $ x(t)\in \partial (Q_1+ y(t))$,\vsm\vsm\vsm
\begin{equation} \overline M\hsmm :=\hsmm \min_{|\zeta|=1}\left\{\max_{u\in U}\{\langle \zeta, f(x(t),u)\rangle\}\hsmm -\hsmm \min_{v\in V}\{\langle\zeta, v\rangle\}\right\}\hsmm,\mbox{ and }\overline m\hsmm :=\hsmm \max_{|\zeta|=1}\left\{\min_{u\in U}\{\langle \zeta, f(x(t),u)\rangle\}\hsmm - \hsmm \max_{v\in V}\{\langle\zeta, v\rangle\}\right\}.\nonumber	\end{equation}
\item[H6] The set ${\bf \Psi}:=\{(T,y,x,u): (T,y)$ is feasible to $ P_H(x,u)$, and $(x,u)$ is feasible to $P_L(T,y)\}$ is nonempty, and $(T^*,y^*,x^*,u^*)$ is not an isolated point in ${\bf \Psi} $.
\end{itemize}

\begin{remark}\label{remark2} Assumption H4, and the need to consider a truncated cone in the lower level sweeping dynamics as defined in H5 are crucial to preserve the bilevel structure of the problem. Given the data of our problem, it is straightforward to conclude that such $M > 0$ exists since $\overline M(x(t),y(t))\hsmm >\hsmm \overline m(x(t),y(t)) $ for all $t  \hsmm\in\hsmm [0,T]$. Moreover if, for some $t  \hsmm\in\hsmm [0,T]$, $M \hsmm >\hsmm\overline M(x(t),y(t))$, then, for any $v(t)\hsmm\in\hsmm V$, the pair $(f, Q_1+y(t))$ is strongly invariant, in a certain neighborhood of $t$ and the lower level solution set would be an inactive constraint to the upper level problem in this neighborhood. On the other hand, if $M\hsmm  <\hsmm \overline m (x(t),y(t))$ and $x(t)\hsmm\in\hsmm \partial (Q_1+y(t))$, then feasible solutions to the lower level problem from $t$ onwards may not exist, and the set of feasible control processes to $(P_H)$ is empty.
\end{remark}
\begin{remark}\label{remark2+} Assumption H4 is very strong as it amounts to full controllability of the drift component of the dynamics. We can dispense with it by considering the following assumption

\noindent $[\overline{ \mbox{H4}}] $ Velocity sets $f(x,U)$ and $V$ are such that $\exists$ a constant $M\hsmm >\hsmm 0$ satisfying the conditions in H5.

\end{remark}
\begin{remark}\label{remark3}  The first part of assumption H6 amounts to the natural existence of compatible feasible control processes to problems at both levels. The second part is also quite natural, and it is required for technical reasons involving the used proof technique.
\end{remark}

For a $(T,v)$ admissible to $(P_H)$, let $\varphi:\R \hspace{-.05cm}\times\hspace{-.05cm} L_2([0,T],\R^n) \to \R$ be the value function of $(P_L)$ on the functional parameters of $(P_H)$.
\begin{equation}  \dis \varphi(T\hsmm,v)=\hsm\min_{(x(0),u_L) \mbox{ feasible to }(P_L)} \hsm\left\{\hsmm J_L(x(0),u_L;T,y)\hsm :  y(t)\hsmm =\hsmm y_0\hsmm  +\hsmm \int_0^t \hsmm \hsm v(s)ds \hsmm \right\}. \label{value function-upper}\end{equation}

Some considerations are in order. Given a pair $(T,y)$ for $(P_H)$, let $ \Phi_L (T,y)$ be the set of feasible control processes for $(P_L)$. Due to the specific structure of the overall problem, notably, the state constraint (\ref{eq:expression of h_L}), and the hypotheses on its data, we have that, $\forall\;(T,y)$ such that $ \Phi_L (T,y)\neq\emptyset$\footnote{As mentioned before, the truncated normal cone in the dynamics of $(P_L)$ entails that the complementary set of such $(T,y)$'s is not empty.}, the set-valued mapping $\Phi_L(\cdot,\cdot)$ is Hausdorff Lipschitz continuous, and, thus, the finite valued map $(T,v)\to \varphi(T,v)$ is Lipschitz continuous in the topology induced by a $\R\hsmm\times\hsmm L_2([0,T],\R^n)$-norm.
By using the value function $\varphi$, the articulation of $(P_H)$ and $(P_L)$ can be expressed in the flattened optimal control problem:
\begin{eqnarray}(P_F)\mbox{ Minimize }&\mbox{\hspace{.2cm} }& \nonumber  J_H(T,y,x(0),u_L):=T,  \\
\mbox{subject to }
&&\dot y= v, \;\;  \dot x\in f(x,u) - N^M_{Q_1+y}(x)\qquad [0,T]-\mbox{a.e.},\label{dynamics-F}\\
&& y(0)= y_0, \; \; y(T) \in \bar E, \;\; x(0) \in Q_1 + y_0, \nonumber \\
&&  v \in \mathcal{V}, \;\; u\in \mathcal{U}, \quad  h_H(y) \le 0 , \; \; h_L(x,y) \le 0  \;\; \mbox{on }[0,T],\;\;\mbox{ and} \nonumber \\
&& \int_0^T\hsm |u_L(t)|^2dt -\varphi(T,v) \leq 0.\label{eq:value function ineq}
\end{eqnarray}
Here, and from now on, $  |u_L(t)|^2$ stands for $ |u(t)|^2+  |u_0(t)|^2$ as in the specification of the cost function of $(P_L(T,y))$.
Before pursuing, consider that for a quadruple $(T, y,x,u_L)$ jointly feasible for $(P_H)$, and $(P_L)$, we have the following result.

\begin{proposition}\label{prop: existence of solution bilevel}
Assume that H1-H6 hold. Then, there exists a solution to $(P_F)$.	
\end{proposition}
This is a direct consequence of merging  \cite[Theorem 4.3]{khalil2019bilevelsweeping} with \cite[Theorem 1]{de2019optimal}. Indeed, while
\cite[Theorem 1]{de2019optimal} yields the existence of an optimal solution to  $(P_L)$ if it has at least one admissible process, \cite[Theorem 4.3]{khalil2019bilevelsweeping} guarantees the existence of an optimal solution to $(P_F)$.

Problem $(P_F)$ presents two key challenges: the dynamics (\ref{dynamics-F}), and, the inequality constraint (\ref{eq:value function ineq}). Indeed, the dynamics (\ref{dynamics-F}) are not Lipschitz continuous in $x$, and, thus, standard necessary optimality conditions are not applicable. The difficulty associated with inequality (\ref{eq:value function ineq}) results from the fact that constraint qualifications, such as Mangasarian-Fromovitz or linear independence constraint qualifications, are too strong to hold. While the first challenge is addressed by considering a proper sequence of approximating problems with dynamics Lipschitz continuous in $x$, the second one requires an additional assumption. Like in \cite{ye1995optimality,ye1997optimal}, we consider the partial calmness property which allows to derive an equivalent problem without this constraint (\ref{eq:value function ineq}) but whose cost functional features the associated exact penalization term.
\begin{remark}\label{remark2}
Nondegeneracy of the necessary conditions of optimality due to the presence of state constraints of the type (\ref{eq:expression of h_H}), and (\ref{eq:expression of h_L}) has long drawn the attention of researchers \cite{A1985,A1989,clarke1990optimization,AS1995,vinter2010optimal}. These references, and \cite{AK2020,AK2016,Mordukhovich2006} reveal that the type and nature of regularity and controllability assumptions on the data of the problem to ensure nondegeneracy has been increasing in sophistication. Given the specific nature of the data of $(P_F)$, one immediately concludes that constraints (\ref{eq:expression of h_H}), and (\ref{eq:expression of h_L}) do not entail the degeneracy of the conditions of the main result of this article.
\end{remark}
Now, we recall the concept of partial calmness. Let $(T^*,y^*,x^*,u_L^*,v^*)$ be a minimizer to $(P_F)$.
\begin{definition}[Partial Calmness] \label{def:partial calmness} $(P_F)$ is called partially calm at $(T^*,y^*,x^*,u_L^*,v^*)$ with modulus $\rho$ if $\exists\; \rho \ge 0$ such that, for any $(T,y,x,u_L,v)$ feasible to $(P_F)$, the following inequality holds
\begin{equation}  J_H(T,y;x(0),u_L) \hsmm\geq\hsmm J_H(T^*,y^*;x^*(0),u_L^*)\hsmm -\hsmm \rho  \hsmm \left(\hsmm\int_0^T\hsm\hsm |u_L(s)|^2ds\hsmm -\hsmm \varphi(T,v)\hsmm\right).\nonumber\end{equation}
\end{definition}
Note that partial calmness is not really a constraint qualification as it compounds properties of the lower level solution set with the properties of the upper problem cost functional. Moreover, for our problem, this issue is compounded by the existence of state constraints jointly on both levels. As noted in \cite{mehlitz2021noteonpartialcalmness,henrion2011calmness}, the partial calmness is, in general, a very restrictive property. However, it has been shown in \cite[Proposition 5.1]{ye1995optimality} that the partial calmness property follows from the fact that the solution to the lower level problem is a uniformly weak sharp minimum. Now, we cast this property in terms of our problem's data. Let $(T^*,y^*,x^*,u_L^*)$ be a solution to $(P_F)$.
\begin{definition}[Uniform Weak Sharp Minimum] \label{def:UWSM} The lower level optimal control problem has a uniform weak sharp minimum around $(x^*(0),u_L^*)$ at $(T^*,y^*)$, if $\exists$ $\delta > 0 $, and $\alpha > 0$ such that\vsm\vsm
\begin{equation}  J_L(x(0),u_L;T,y) - \bar\varphi(T,y)  \geq \alpha d_{\Psi(T,y)}(x(0),u_L),\nonumber\vsm\vsm\end{equation}
$\forall\; (T,y)\in\delta B_{\R\times AC}(T^*,y^*)$, and $\forall\; (x(0),u_L)\in\delta B_{\R^n\times L_\infty}(x^*(0),u_L^*)$, being $\dis d_A(a):= \inf_{\bar a\in A} \{| a-\bar a|\}$ the distance of the point $a$ to the set $A$.\

\noindent Here, $\bar\varphi$ is the value function of the lower level problem $\varphi$ defined above but now as a function of $(T,y)$ instead of $(T,v) $ with $\dis y(t)=y_0+\int_0^t v(\tau)d\tau$, $\Psi(T,y)$ is the solution set of the lower level problem for any pair $(T,y)$ such that $(T,y,x,u_L)\in{\bf \Psi}$ and $B_{X\times Y}(a,b)$ is the unit ball in $X\times Y$ centered in $(a,b)$.\vsm\vsm
\end{definition}

It can easily be shown that, for our problem that, if $(T^*,y^*,x^*,u_L^*,v^*)$ is a minimizer of $(P_F)$, then Definition \ref{def:UWSM}  is satisfied.
Take $(T^*,y^*,x^*,u_L^*,v^*)$ to be a minimizer of $(P_F)$. From $H6$, there are processes $(T,y,x(0),u_L)$, and $(T,y,\bar x(0),\bar u_L)$ in a $\varepsilon$-neighborhood (in the appropriate spaces) of $(T^*,y^*,x^*(0),u_L^*)$ for any arbitrary small $\varepsilon\hsmm >\hsmm 0$ such that $$\vsm J_L(x(0),u_L;T,y)\hsmm > \hsmm J_L(\bar x(0),\bar u_L;T,y).\vsm$$ Moreover, for a given $\varepsilon$, we have, also from $H6$, that, for some $\delta\hsmm >\hsmm 0$, we may take the above pairs to be arbitrary such that $(T,y)\hsm \in\hsm \delta B_{\R\times AC}(T^*,y^*)$, and both $(x(0),u_L)$, and $(\bar x(0),\bar u_L)$ are in $\delta B_{\R^n\times L_\infty}(x^*(0),u_L^*)$. Note also that, from the continuity of the cost functional, and the compactness of the feasible set of $P_L(\cdot,\cdot;T,y)$, as well as the continuity of the joint state constraint on $x$ and $y$, the existence of a solution is clear. Let us choose $(\bar x(0), \bar u_L)$ as above to be a solution to $P_L(\cdot,\cdot;T,y)$, i.e., $\dis \bar \varphi(T,y) \hsm =\hsm \int_0^T\hsm |\bar u_L(\tau)|^2 d\tau$. Now, by examining the structure of the lower level optimal control problem, it is also clear that $\exists\;\gamma \hsmm > \hsmm 0$ such that $\|u_L\hsmm -\hsmm \bar u_L\| > \gamma \|(x(0),u_L)\hsmm -\hsmm (\bar x(0),\bar u_L)\|$, and, as a consequence, also some $\beta\hsmm >\hsmm 0$ such that $ \|u_L\|^2\hsmm -\hsmm \|\bar u_L\|^2  \geq \beta \| u_L\hsmm -\hsmm \bar u_L \| $ $\forall\; u_L$ satisfying the above conditions. Thus, by choosing $\alpha\hsmm =\hsmm \beta\gamma $, we have
\vspace{-.3cm}
$$ J_L(x(0),\hsmm u_L\hsmm;T,y) - J_L(\bar x(0),\hsmm \bar u_L\hsmm;T,y) = \|u_L\|^2\hsm -\hsmm\|\bar u_L\|^2\geq \beta \| u_L\hsmm-\bar u_L \| \geq \alpha d_{\Psi(T,y)}(x(0), u_L). $$

\vspace{-.6cm}
In order to state the necessary optimality conditions in the Gamkrelidze's form for the bilevel optimal control problem $(P_H)$ articulated with $(P_L)$ (or equivalently the single level problem $(P_F)$), let us define

\vspace{-1cm}
\begin{eqnarray}
H_H(y,x,v,u,q_H,q_L,\nu_H,\nu_L, r)&:=& \langle q_H\hsmm - \hsmm\nu_H(y\hsmm -\hsmm q_0),v\rangle\hsmm +\hsmm \nu_L \langle x\hsmm -\hsmm y,v\rangle\hsmm - \hsmm r|u|^2 \nonumber\\
&& \hspace{1cm} +\langle q_L\hsmm-\hsmm \nu_L(x\hsmm -\hsmm y),f(x,u)\rangle\hsmm +\hsmm \sigma(y,x,q_L,\nu_L,r), \nonumber\vsm
\label{expression of higher level hamiltonian}\end{eqnarray}
\vspace{-1cm}

\noindent where $y,\, x,\,v,\, q_H,\, q_L $ take values in $\R^n$, $u$ in $ \R^m$, $\nu_H,\,\nu_L$ in $\R$, and $ r$ is a nonnegative scalar, and, by using the alternative truncated normal cone expression in (\ref{alternate-Normalcone}) $\sigma(y,x,q_L,\nu_L, r)=0$ if $|x-y| < R_1$ or

\vspace{-.6cm}
$$\dis \sigma(y,x,q_L,\nu_L, r) = \sup_{u_0\in [0,1]} \left\{\langle q_L\hsm -\hsm\nu_L(x-y),\bar f(x,y)(x-y)u_0\rangle- r u_0^2\right\}.$$
\vspace{-1cm}

\noindent Observe that, for $|x-y|=R_1$, function $\sigma(y,x,q_L,\nu_L, r)$ is explicitly expressed by

\begin{equation} \label{sigma} \sigma(y,x,q_L,\nu_L, r) = \left\{ \begin{array}{ll} 0 & \mbox{ if } \tilde\sigma (q_L,\nu_L, x,y, R_1)\leq 0\\
\frac{1}{r}\frac{M^2}{4R_1^2}\tilde\sigma^2(q_L,\nu_L, x,y, R_1) & \mbox{ if } 0 <\tilde\sigma(q_L,\nu_L, x,y, R_1)  \leq r\frac{2R_1}{M}\\
\frac{M}{R_1}\tilde \sigma^2(q_L,\nu_L, x,y, R_1) - r & \mbox{ if } \tilde\sigma(q_L,\nu_L, x,y, R_1)  \geq r\frac{2R_1}{M},
\end{array}\right. \end{equation}
where $\tilde\sigma(q_L,\nu_L, x,y, R_1):= \nu_LR_1^2-\langle q_L,x-y\rangle$.

Now, we state our necessary conditions of optimality to $(P_F)$.

\vspace{-.2cm}

\begin{theorem}\label{theorem:main thm} Let H1-H6 hold, and $(T^*,y^*,x^*,v^*,u^*)$ be a solution to $(P_F)$. Then, there exists a set of  multipliers $(q_H,q_L,\nu_H,\nu_L,\lambda, r)$, and a constant $ c\hsmm\in\hsmm \R$, with $q_H$, and $q_L$ in $AC([0,T^*];\R^n)$, non-increasing $\nu_H$, and $\nu_L$ in $N\hsmm BV([0,T^*];\R)$, being $\nu_H$ constant on $\{t\hsm\in\hsm[0,T^*]\hsm:\hsm |y-z|\hsm >\hsm R_1\,
\forall z\hsm\in \hsm \partial Q\}$, and $\nu_L$ constant on $\{t\hsm\in\hsm[0,T^*]\hsm:\hsm |y-x|\hsm <\hsm R_1\}$, and non-negative numbers $\lambda$, $r$, with $r\hsmm =\hsmm \lambda\rho $, being $\rho$ the partial calmness modulus, satisfying the following conditions:

\vspace{-.2cm}
\begin{itemize}
		\item[1.] Nontriviality. $ \|(q_H,q_L)\|_{L_\infty} + \|(\nu_H,\nu_L)\|_{TV}+\lambda + r \neq 0.$ \vspace{.1cm}
		\item[2.] Adjoint equations.
\vspace{-.3cm}
		\begin{eqnarray}
		&&\hspace{-1cm}  \nonumber-\dot q_H(t) \in \partial_y H_H(y^*,x^*,v^*,u^*,q_H,q_L,\nu_H,\nu_L, r) = - ( \nu_H(t) +\nu_L(t))v^*(t)+ \nu_L(t)f(x^*(t),u^*(t))\\
		&& \hspace{6cm}+ \partial_y \sigma(y^*(t),x^*(t),q_L(t),\nu_L(t),r)\; \;\; [0,T^*]\hsmm -\hsmm\mbox{a.e.},\label{adjoint_H}\\
     	&& \hspace{-1cm} \nonumber -\dot q_L(t)\in  \partial_x H_H(y^*,x^*,v^*,u^*,q_H,q_L,\nu_H,\nu_L, r)  =  \partial_x\langle q_L(t)-\nu_L(t)(x^*(t)- y^*(t)),f(x^*(t),u^*(t)) \rangle\\
		&& \hspace{4.5cm}+ \nu_L(t) v^*(t)\hsmm +\hsmm \partial_x \sigma(y^*(t),x^*(t),q_L(t),\nu_L(t),r) \;\; [0,T^*]\hsmm -\hsmm\mbox{a.e.}. \label{adjoint_L}\end{eqnarray}
\vspace{-1cm}
		\item[3.] Boundary conditions.
\vspace{-.3cm}
		\begin{eqnarray*} &&\hspace{-1cm}  q_H(0)\hsmm\in\hsmm\R^n, \;\;\;\; q_H(T^*) \hsmm \in\hsmm - N_{\bar E}(y^*(T^*)) \hsm +\hsm\nu_H(T^*)(y^*(T^*)\hsm -\hsm q_0)\hsm  -\hsm \nu_L(T^*)(x^*(T^*)\hsm -\hsm y^*(T^*)),\nonumber \\
        && \hspace{-1cm}q_L(0)\in  N_{Q_1+y_0}(x^*(0)) + \nu_L(0)(x^*(0)-y_0), \;\;\;\;  q_L(T^*)  \hsmm = \hsmm \nu_L(T^*)(x^*(T^*)-y^*(T^*)). \nonumber
		\end{eqnarray*}
\vspace{-.9cm}
		\item[4.] Conservation law.
\vspace{-.3cm}
		\begin{equation} H_H(y^*,x^*(t),v^*(t),u^*(t),q_H(t),q_L(t),\nu_H(t),\nu_L(t), r) = \lambda + r c, \;  \forall t\in [0,T^*]. \label{time-transversality}\end{equation}
\vspace{-.9cm}
		\item[5.] Maximum condition on the lower level control. $u^*(t)$ maximizes on $U$, $[0,T^*]$-a.e., the mapping
\vspace{-.3cm}
		\begin{equation}\hspace{-.3cm} u\to  \langle q_L(t)\hspace{-.03cm} - \hspace{-.03cm} \nu_L(t)(x^*(t)\hspace{-.03cm} - \hspace{-.03cm} y^*(t)),f(x^*(t),u)\rangle - r |u|^2.
		\label{max-u}\end{equation}
\vspace{-.9cm}
		\item[6.] Maximum condition on the upper level control.
\vspace{-.3cm}
		\begin{equation}  -q_H(t)+\nu_H(y^*-q_0)-\nu_L (x^*- y^*)\in  -r \partial_v^C\varphi(T^*,v^*)+  N_{\cal V}(v^*). \label{max-vw}\end{equation}
	\end{itemize}
\vspace{-.3cm}

\noindent Now, we give an explicit expression for $\partial_v^C\varphi(T^*,v^*)$ in condition 6. Denote by $\Psi (T,y)$ the set of solutions to $(P_L)$ for a given feasible pair $(T,y)$, and by $H_L$ its Hamilton-Pontryagin function given by

\vspace{-.9cm}
\begin{eqnarray*}
H_L(y, x, v,u,\bar p,\bar \mu, \bar \lambda) & := & \langle p_H\hsmm -\hsmm\mu_H (y- q_0), v\rangle\hsmm +\hsmm\mu_L\langle x-y,v \rangle\\
&& \hspace{2cm} +\langle p_L\hsm -\hsm \mu_L(x\hsm -\hsm y), f(x, u)\rangle\hsm -\hsm\bar\lambda |u|^2 + \bar \sigma(y,x,p_L,\mu_L, \bar \lambda), \nonumber
\end{eqnarray*}

\vspace{-.6cm}
\noindent where $y$, $x$, $v$, $p_H$, and $ p_L $ take values in $\R^n$, $u_L$ in $ \R^m$, $\mu_H$, and $\mu_L$ in $ \R$, and $ \bar \lambda $ is a nonnegative scalar, $ \bar p:=(p_H, p_L) $, $ \bar \mu:=(\mu_H,\mu_L) $, $\bar \sigma(y,x,p_L,\mu_L, \bar \lambda)$ either takes the value $0$ if $|x\hsmm -\hsmm y| \hsmm <\hsmm R_1$, or the value $\dis \sup_{u_0\in [0,1]} \hsm\{\langle p_L\hsm -\hsm\mu_L(x-y),\bar f(x,y)(x\hsmm -\hsmm y)u_0\rangle\hsmm-\hsmm \bar \lambda u_0^2\} $ if $|x-y| = R_1$ (in this case, $\bar \sigma(\cdot)$ takes an explicit form analogue to (\ref{sigma})), and the set valued mapping $ \partial^C_v \varphi(\cdot,\cdot)$ is given by
\vsm\vsm
\begin{eqnarray*}\partial^C_v \varphi(T^*,v^*)&:=&\hsm {\textrm{co }}\hsm\hsm\hsm\hsm \bigcup_{x^*\in\Psi (T^*,v^*)}\hsm \hsm\hsm\hsm {\Large \{}\zeta_2 \hsm\in\hsm L_2([0,T^*]\hsm :\R^{n}) : \exists\, (\bar p,\bar \mu, \bar \lambda) \hsm\in\hsm AC \hsm\times\hsm N\hspace{-.04cm}BV\hsm \times\hsm \R^+ \mbox{ s. t. } \|\bar p\|_{L_\infty}\hsmm +\hsmm \bar \lambda\hsmm +\hsmm \|\bar \mu\|_{TV} \hsmm\neq\hsmm 0, \vspace{-.4cm}\nonumber \\
	&& \hspace{1.5cm} \bar\mu \mbox{ monotonically non-increasing, with }\mu_L\mbox{ and }\mu_H \mbox{constant on, respectively,}\nonumber\\
    && \hspace{1.5cm} \{t\hsm\in\hsm[0,T^*]\hsm:\hsm |y-x|\hsm <\hsm R_1\} \mbox{ and }\{t\hsmm\in\hsmm[0,T^*]\hsm:\hsm |y-z|\hsmm >\hsmm R_1\, \forall z\hsmm\in \hsmm\partial Q\}, \nonumber\\
	&& \hspace{1.5cm} (-\dot{\bar p},\dot y^*,\dot x^*)\hsmm\in\hsmm\partial^C_{(y, x, \bar p)}H_L( y^*, x^*, v^*,u^*,\bar p,\bar \mu, \bar \lambda)\;\;\mbox{a.e.}, \;\; (\bar p(0),\bar p(T^*))\hsmm\in\hsmm\bar P_L,\;\;\mbox{and}\nonumber\\
	&& \hspace{1.5cm} -\bar \lambda \zeta_2 \in p_H\hsmm -\hsmm\mu_H(y^*\hsmm -\hsmm q_0)\hsmm +\hsmm \mu_L(x^*\hsmm -\hsmm y^*) \hsmm +\hsmm N_{\mathcal{V}}(v^*){ \Large\} }.\nonumber
\end{eqnarray*}

\vsm\vsm\vsm
\noindent Here,
\vspace{-1.2cm}
\begin{eqnarray*} \hspace{.9cm}
{\bar P}_L&\hsm :=\hsm &\{(\bar p(0),\bar p(T^*)) \hsm : p_H(0)\hsm \in\hsm \R^n\hsmm , \,p_L(0)\hsm\in\hsm N_{Q_1+y_0^*}(x^*(0))\hsm +\hsm \mu_L(0) (x^*(0)\hsm -\hsm y_0^*), \;p_H(T^*)\hsm\in\hsm - N_{\bar E}(y^*(T^*))\\
&& \hspace{.6cm} + \mu_H(T^*) (y^*(T^*)\hsmm -\hsmm q_0)\hsmm-\hsmm \mu_L (T^*) (x^*(T^*)\hsmm  -\hsmm y^*(T^*)),\, p_L(T^*) \hsmm =\hsmm  \mu_L(T^*) (x^*(T^*)\hsmm -\hsmm y^*(T^*))\}.
 \end{eqnarray*}
\end{theorem}
The proof of Theorem \ref{theorem:main thm}, including the computation of $\partial^C_v \varphi(T^*,v^*)$, follows in the next section.
\section{Proof of the Main Theorem}\label{sec:proof}
The proof is organized into five steps:
\begin{itemize}
\item[Step 1.] Here, the key challenge of the sweeping process at the lower level problem is addressed. Indeed, the main difficulty is the lack of Lipschitz continuity of the right-hand side of the dynamics of $(P_L)$ with respect to the state variable. In order to overcome this difficulty, we construct a sequence of auxiliary problems $\{(P_L^k)\}$ approximating $(P_L)$, being the truncated normal cone in the dynamics of each $(P_L^k)$ replaced by a mapping depending on the parameter $k$ and Lipschitz continuous w.r.t. the state variable. Note that our construction differs from the one in \cite{de2019optimal,zeidan2020sweeping} where the usual normal cone is considered. Moreover, in opposition to  work in \cite[Lemma 1]{de2019optimal}, or \cite[Lemma 4.2]{zeidan2020sweeping}, the state constraint cannot be discarded from the formulation of our approximation to the lower level problem.

In this step, we also provide an explicit expression for $\partial^C_{v} \varphi^k$, where $\varphi^k$ is the value function for $(P_L^k)$.
\item[Step 2.] Now, two substeps are combined: (a) Flattening of the approximating bilevel problem, and (b) Application of the Ekeland's
variational principle. Substep (a) allows to obtain a standard ``single-level'' approximating optimal control problem. The flattening consists in replacing the original pair of coupled optimal control problems $(P_H)$, and $(P_L^k)$ by another one constructed by adding to $(P_H)$ the constraints of $(P_L^k)$, and an additional constraint involving the value function of $(P_L^k)$ that depends on $(T,y)$. Since the solution to each $(P_L^k)$ is not known, Substep (b) is required. It consists in the application of Ekeland's variational principle to ensure the existence of a solution to a suitably perturbed version of the approximating flattened problem $(P_F^k)$. By a suitably perturbed version, we mean that the solutions to the approximating sequence of problems converge in a proper sense (to be defined below) to the one of the original optimal control problem.
\item[Step 3.] Here, the degeneracy caused by the constraint (\ref{eq:value function ineq}) involving the value function associated with the lower level problem as detailed in Step 2 (a) is handled by using the partial calmness property which holds from the fact that the solution is a local uniform weak sharp minimum. Due to the partial calmness property, an exact penalization technique allows the construction of a related optimal control problem where this additional constraint is absorbed in the cost functional as an exact penalty term.
\item[Step 4.] Now, the maximum principle of Pontryagin in the Gamkrelidze's form of \cite{arutyunov2011maximum} with nonsmooth data, \cite{KLP-paperIEEE-LCSS2020}, is applied to the perturbed approximating single-level problem. We observe that, since the velocity set of the approximating problem includes the one of the original problem, no incompatibility between state and endpoint constraints emerges. Besides the reasons pointed out in Section \ref{sec: introduction} for choosing the Gamkrelidze's form, there are significant technical advantages due the fact that, in this form, the state constraints multipliers are monotonic functions of bounded variation instead of mere Borel measures as it is the case of the Dubovitskii-Miliutyn form.
\item[Step 5.] Here, we pass to the limit in the necessary conditions of optimality established in Step 4. By using standard compactness
results, the necessary optimality conditions of our main result are recovered.
\end{itemize}
As we go through these five steps, most of the relevant intermediate results will be only stated in order to simplify the presentation, being the corresponding proofs presented in the Appendix.

\vspace{-.6cm}
\subsection{Lower Level Dynamics Approximation, and Computation of $\partial^C \varphi^k$}\label{subsec:approx lower level}

\noindent a) Lipschitz Continuous Approximation to the Lower Level Dynamics

Here, a sequence of conventional - in the sense of dynamics being Lipschitz continuous on the state variable - control processes approximating a feasible sweeping control process to $(P_L)$. Approximations of this type have been considered in \cite{de2019optimal,zeidan2020sweeping}.
However, since, we have a bounded truncated normal cone instead of the usual normal cone, and a time varying set $Q_1+y(t)$ instead of a constant one, a construction scheme significantly different of the previous ones is required.

By recalling the representation of the lower level dynamics, $\dis \dot x = f(x,u) + \bar f(x,y)u_0(x-y) $, where $u_0\hsmm\in\hsmm {\cal U}_0\hsmm :=\hsmm L^1([0,T];[0,1]) $ is a scalar control, and $\bar f:\R^{2n}\to\R$ is defined in (\ref{alternate-Normalcone}), we now consider the approximate lower level control system $(D^k(T,y))$, referred to by $(D^k)$,
$$
(D^k)\qquad \left\{ \begin{array}{l} \dot x =f^{k}(x,y,u,u_0)\;\; [0,T]-\mbox{a.e.}\quad x(0)\in Q_1+y_0,\\
h_L(x,y) \le 0,  \;\; \forall t\in [0,T], \quad u\in \mathcal{U}, \qquad u_0\in \mathcal{U}_0,
\end{array}\right. \qquad\mbox{ where}
$$
\vsm\vsm\vsm
\begin{equation} \label{approx of normal cone}f^k(x,y,u,u_0):=f(x,u)- u_0 c ( \gamma_k,x,y)(x-y),\end{equation} being $\dis c(\gamma_k,x,y) := \min \left\{\hspace{-.07cm}\frac{M}{R_1}, \gamma_k e^{\gamma_k h_L(x,y)}\hspace{-.07cm}\right\}$, $h_L(x,y)$ as in (\ref{eq:expression of h_L}), and  $\{\gamma_k\}$ is a sequence such that
\begin{equation}\label{eq:condition on gamma_k}
\lim\limits_{k\to\infty}\gamma_k =\infty, \mbox{ and, for all } k, \;\; \gamma_k>\frac{M}{R_1}.
\end{equation}
Consider the $k$-approximate lower level problem defined by
$$(P_L^k)\mbox{ Minimize }  J_L(x(0),u_L; T,y):=\int_{0}^{T} \hsm |u_L(t)|^2 dt  \mbox{ subject to } (D^k).$$
Define $ {\cal F}_L^k (T,y)\hsmm :=\hsmm \{(x,u,u_0)\hsmm \in\hsmm AC([0,T];\hsmm\R^n)\hsmm\times\hsmm{\cal U}\hsmm\times\hsmm \mathcal{U}_0\hsmm : (x,u,u_0) \mbox{ feasible for } (D^k)  \} $ where $(T,y,x,u)$ is feasible for $(P_F)$.
\begin{proposition}\label{proposition: nonempty feasible set}
For $\gamma_k$ satisfying (\ref{eq:condition on gamma_k}), and $(T,y)\hsmm \in\hsmm \R\hsm \times \hsm \R^n$ such that
 ${\cal F}_L (T,y)\neq \emptyset $, ${\cal F}_L^k (T,y)\neq\emptyset$.
\end{proposition}
\begin{proof}. Take a pair $(T,y)\in\R\hsmm\times\hsmm \R^n$ such that ${\cal F}_L (T,y)\neq \emptyset$. Standard results for ordinary differential equations with dynamics Lipschitz continuous with respect to the state variable $x$ asserts the existence of a solution for any given feasible control pair $(u,u_0)$. Moreover, due to the fact that $\dis \gamma_k>\frac{M}{R_1}$, and, by construction, we have   that, for any given $(T,y, x)$, the velocity set associated with the dynamics of $(D^k)$ contains the one of the dynamics of $(P_L)$\footnote{More generally, for any given $x\hsmm\in\hsmm\R^n$, if $\dis \tilde k> \bar k> \frac{M}{r_1}$, the velocity set of $(D^{\bar k})$ is contained in the one of $(D^{\tilde k})$}. Thus, the existence of a control process $(x_k,u_k, u_{0,k})$ feasible to $(D^k)$ satisfying the constraints of $(P_L^k)$ is guaranteed.
\end{proof}
Take $\overline T > 0$, and $\Delta(\cdot.\cdot)\hsmm :\hsmm (\mathcal{V}\hsmm\times\hsmm\mathcal{U} \hsmm\times\hsmm\mathcal{U}_0)\times (\mathcal{V}\hsmm\times\hsmm\mathcal{U}\hsmm\times\hsmm\mathcal{U}_0)\to [0,\infty)$, with $\mathcal{V}, \mathcal{U}$, and $\mathcal{U}_0$ defined on $[0,\overline T]$, given by
\begin{equation}\label{expression of Ekeland metric}
	\Delta(\bar\theta_1,\bar\theta_2):=\| \bar\theta_1-\bar\theta_2\|_{L_1},
\end{equation}
and let ${\cal S}:= \R^+ \hsmm\times\hsmm AC([0,\overline T];\R^n)\hsmm\times\hsmm AC([0,\overline T];\R^n)\hsmm \times\hsmm{\cal V}  \hsmm \times \hsmm {\cal U} \hsmm\times \hsmm {\cal U}_0$. In what follows, let $\dis \overline T =\max_{i\in \N}\{ T, T_i\} $, and, put $\bar\theta_i=0$ if $\bar\theta_i $ is not specified in some subinterval $[T_i,\overline T]$.

From the proposed approximation scheme, and due to Proposition \ref{proposition: nonempty feasible set}, the following result holds.
\begin{proposition}\label{prop: uniqueness solution} Let $(T,y,x,v,u,u_0)\hsmm\in\hsmm {\cal S}$ be such that $(x,u,u_0)\hsmm\in\hsmm {\mathcal F}_L(T,y)$. Then, there exists a sequence $\left\{(T_k,y_k, x_k,v_k, u_k,u_{0,k})\right\}_{k=1}^\infty$ with elements in ${\cal S}^k$, being ${\cal S}^k$ the space ${\cal S}$ defined on $[0,T_k]$, satisfying:
\begin{itemize}
\vspace{-.1cm}
\item[1)]  $|T\hsmm -\hsmm T_k| +|x_k(0)\hsmm -\hsmm x(0)|+\Delta((v_k,u_k,u_{0,k}),(v,u,u_0)) \to 0 $ as $k\to \infty$,
\item[2)]  $y_k(0)=y_0$, and $\dot y_k = v_k$ $ {\cal L}$-a.e. on $[0,T_k]$, and
\item[3)]  $(x_k,u_k,u_{0,k})\in {\mathcal F}_L^k(T_k,y_k)$, with $(D^k)$ defined on $[0,T_k]$.
\end{itemize}
\vspace{-.1cm}
Moreover, there exists a subsequence (we do not relabel) of $\{x_k\}$ converging uniformly to an arc $x$ which is the unique solution to
\vspace{-.4cm}$$ (D)\quad\dot x \in f(x,u)- N_{Q_1+y}^M(x) \;\; [0,T]-\mbox{a.e.}, \;\; x(0) \in Q_1+y_0, \quad u\in {\cal U}, \quad h_L(x,y) \le 0\;\; \forall \; t \in [0,T].
$$
\end{proposition}
\vspace{-.3cm}Since the proof of Proposition \ref{prop: uniqueness solution} is long and involves standard arguments, we include it in the Appendix.
\vspace{.05cm}

\noindent b) Computation of $\partial^C \varphi^k$

Denote by $\varphi^k$ the value function for $(P_L^k)$ in (\ref{value function-upper}). Since it is needed to establish the necessary conditions of optimality, we shall compute its limiting subdifferential, denoted by $\partial_v^C \varphi^k$ (whose limit, $\partial_v^C \varphi$, appears in the maximum condition on the upper level problem in Condition 6. of Theorem \ref{theorem:main thm}). For this purpose, we apply the idea in \cite[Theorem 2.3]{ye1997optimal}, albeit in the context of the Hamilton-Pontryagin function in the Gamkrelidze's form given by (\ref{hamilton pontryagin function}) (cf. \cite{KLP-paperIEEE-LCSS2020}). Following the idea in \cite{ye1997optimal}, we notice that the arguments of the value function inherited from the upper level problem are given in terms of control functions.
For this reason, it is convenient to cast our problem in an equivalent one on a fixed time interval $[0,T^*]$ by a standard time variable change, giving rise to an additional scalar nonnegative control and state component, i.e., $ t\hsmm\in\hsmm AC([0,T^*];\R^+)$, with $t(0)= 0$, and $ \omega\hsmm\in\hsmm L_2([0,T^*];\R^+)$, related by
\begin{equation}\label{omega definition} t(\tau)=\int_0^\tau \omega (s)ds . \end{equation}
This allows the value function in (\ref{value function-upper}) to be expressed as in the context of \cite{ye1997optimal}, i.e., the value function of $(P_L^k)$ on the functional parameters of $(P_H)$, $\varphi^k:L_2([0,T];\R^+)\hsmm\times\hsmm L_2([0,T];\R^n) \to \R$, where
\begin{equation}\varphi^k(\omega,v)=\min \{ J_L(x(0),u_L;T,y)\hsm : (x(0),u_L) \mbox{ feasible for } (P_L^k(T,y))\}, \label{expression of value function}\end{equation}
being $v\hsmm\in\hsmm {\mathcal V}$ the control associated with $y$, and $\omega$ the new control yielding the final time $T$, on $[0,T^*]$.
Now, problem $(P_L^k)$, i.e., $(P_L^k(t(T^*),y))$, can be written (without relabeling) as follows
\begin{eqnarray}
(P_L^k) \nonumber &&\mbox{ Minimize}\; J_L(x(0),u_L;t(T^*),y) := \int_{0}^{T^*} \hsm\hsm |u_L(s)|^2 \omega(s)ds, \nonumber\\
&& \nonumber\mbox{subject to }(\overline D^k) \left\{\begin{array}{ll} \dot x= f^{k}(x,y,u,u_0)\omega, \quad \dot t = \omega,  \quad [0,T^*]-\mbox{a.e.},\\
x(0) \in Q_1\hsmm +\hsmm y_0,  \quad t(0)=0.\\
u\in \mathcal{U},  \quad u_0  \in \mathcal{U}_0,  \quad \omega \in L_2([0,T^*];\R^+). \label{feasible upper level approx} \\	
h_L(x,y) \leq 0 \quad \forall t\in [0,T^*], \end{array}\right.
\end{eqnarray}
where $f^{k}(x,y,u,u_0)$ is as in (\ref{approx of normal cone}), and $(t,y)$ (or, equivalently, the associated control pair $(\omega,v)$) is a solution to
$(P^k_H(x(0),u_L))$ on the fixed time interval $[0,T^*]$, now cast as
\begin{eqnarray*}(P^k_H) \mbox{ Minimize }&&  J_H(t(T^*),y,x(0),u_L):= t(T^*). \nonumber\\
	\mbox{subject to }&& \dot y= v\omega ,  \;\; \dot t=\omega  \;\;\; [0,T^*]-\mbox{a.e.}, \quad y(0)= y_0 \in Q ,  \;\;y(T^*)\in \bar E,  \;\; t(0)=0,\nonumber\\
	&& v \in \mathcal{V},  \;\; \omega \in L_2([0,T^*];\R^+),  \;\;  h_H(y) \le 0  \;\; \forall t \in [0,T^*], \mbox{ and}\\
	&&(t(T^*),y) \mbox{ is s.t. }\exists\mbox{ a solution } (x(0),u_L) \mbox{ to } P_L^k(t(T^*),y). \nonumber
\end{eqnarray*}

\vspace{-.6cm}
The new time parametrization enables to express the functional interdependence between $(P^k_H)$ and $(P_L^k)$ in the variables $(t(T^*),y)$, through the value function $\varphi^k$ as a function of the controls $(\omega,v)$.

Let $\bar p= ( p_H, p_L)$, $\bar \mu = ( \mu_H, \mu_L) $, $(\dot y, \dot x) = (v, f(x,u)\hsmm -\hsmm u_0 c(\gamma_k, y, x)(x\hsmm -\hsmm y))\omega$, and $\omega$ as in (\ref{omega definition}). Since the context is clear, we omit the index $k$ in the specification of the multipliers, the state, and control variables in order to unburden the notation. The Hamilton-Pontryagin function in the Gamkrelidze's form for $(P_L^k)$, is expressed by $H_L^k (\cdot,\omega)\hsm = \hsm \overline H_L^k(\cdot)\omega$, where
\vspace{-.3cm}
\begin{eqnarray} \overline H_L^k(y, x, v,u,\bar p,\bar \mu,\bar  \lambda) &  := & \langle p_H\hsmm -\hsmm\mu_H (y\hsmm -\hsmm  q_0), v\rangle +\mu_L\langle x\hsmm -\hsmm y,v \rangle-\bar \lambda |u|^2 \nonumber\\ && \label{hamilton pontryagin function} \hspace{2cm}+\langle p_L\hsmm - \hsmm \mu_L(x- y), f(x, u)\rangle\hsmm + \hsmm \bar \sigma^k(y,x, p_L, \mu_L,\bar\lambda).
\end{eqnarray}
\vspace{-.34cm}
Here, $\bar \lambda \ge 0$,  and $\bar \sigma^k(y,x, p_L, \mu_L,\bar\lambda) = 0 $ if $|x-y|< R_1$ and, for $|x-y| = R_1$, given by
$$\dis  \left\{\begin{array}{ll} 0 & \mbox{  if }0 \geq \tilde\sigma^k(p_L,\mu_L, x,y,R_1) \\
\dis\frac{1}{4\bar \lambda}\left(c(\gamma_k,x,y)\tilde\sigma^k(p_L,\mu_L, x,y,R_1)\right)^2 & \mbox{  if } 0 \leq \tilde\sigma^k(p_L,\mu_L, x,y,R_1) \leq \frac{2\bar\lambda}{c(\gamma_k,x,y)} \\
\dis c(\gamma_k,x,y)\tilde \sigma^k(p_L,\mu_L, x,y,R_1)-\bar \lambda & \mbox{  if }\tilde\sigma^k(p_L,\mu_L, x,y,R_1)\geq \frac{2\bar\lambda}{c(\gamma_k,x,y)},\end{array}\right.$$
where $\tilde\sigma^k(p_L,\mu_L, x,y,R_1):= \mu_L R_1^2-\langle p_L,x-y\rangle$.
\vspace{-.1cm}
\begin{proposition}\label{subgrad-value}
Let $\Psi^k(\omega,v)$ be the set of optimal solutions to $(P_L^k(T,y))$, where $\dis T\hsmm =\hsm \int_0^{T^*}\hsm\hsm\omega(s)ds $, and let $(x,u)$ be one of such optimal control processes. Assume that H1-H6 are in force. Then, $ \varphi^k$ is Lipschitz continuous in its domain, and
\vspace{-.4cm}
\begin{eqnarray}\partial^C \varphi^k(\omega,v)&\hsm:=&\hsm\mbox{co}\hsm\hsm\hsm\hsmm \bigcup_{x\in\Psi^k (\omega,v)} \hsm\hsm\hsm\hsm\hsmm \{(\zeta_1,\zeta_2) \hsm\in\hsm L_2([0,T^*]\hsm :\R^{n+1})\hsm : \hsmm \exists\, (\bar p,\bar \mu, \bar \lambda) \hsm\in\hsm AC \hsm\times\hsm N\hspace{-.03cm}BV\hsm \times\hsm \R^+ \mbox{ s.t. } \|\bar p\|_{L_\infty}\hsmm +\hsmm \bar \lambda\hsmm +\hsmm \|\bar \mu\|_{TV}\hsmm \neq\hsmm 0,\vspace{-.4cm}\nonumber \\
&& \hspace{.1cm} \bar\mu \mbox{ is non-increasing, with }\mu_L\hsmm \mbox{ and } \mu_H \mbox{constant, respectively, on}\nonumber\\
&& \hspace{.1cm} \{t\hsmm\in\hsmm[0,T^*]\hsm:\hsmm |y-x|\hsm <\hsm R_1\},\hsmm\mbox{ and }\{t\hsmm\in\hsmm[0,T^*]\hsm:\hsmm |y-z|\hsmm >\hsmm R_1\, \forall z\hsm\in \hsm\partial Q\}, \nonumber\\
&& \hspace{.1cm} (-\dot{\bar p},\dot y,\dot x)\hsmm\in\hsmm\partial^C_{(y, x, \bar p)}H_L^k( y, x, v,u,\bar p,\bar \mu, \bar \lambda,\omega)\;\;[0,T^*]\hsmm -\hsmm \mbox{a.e.},\quad (\bar p(0),\bar p(T^*)), \in \bar P_L^k\label{grad-aval-k}\\
&& \hspace{.1cm} -\bar \lambda( \zeta_1, \zeta_2) \in \{\overline H_L^k(y, x, v,u,\bar p,\bar \mu,\bar  \lambda)\} \hsmm\times\hsmm (p_H-\mu_H(y- q_0)+ \mu_L(x- y))\omega + N_{\mathcal{V}}(v)\}. \nonumber
\end{eqnarray}
Here,
 \vspace{-1.2cm}\begin{eqnarray*}\hspace{.7cm} \bar P_L^k &\hsmm =&\hsmm \{(\bar p(0),\bar p(T^*))\hsm : \hsmm p_H(0) \hsm\in\hsm\R^n,\,p_L(0)\hsmm\in\hsmm N_{Q_1}(x(0) \hsmm- \hsmm y_0)\hsm +\hsm \mu_L(0) (x(0)\hsmm -\hsmm y_0), \;p_H(T^*)\hsm\in\hsm - N_{\bar E}(y(T^*)\\
&& \hspace{1.6cm}+ \mu_H(T^*) (y(T^*)\hsmm -\hsmm  q_0)\hsmm -\hsmm  \mu_L (T^*) (x(T^*)\hsmm -\hsmm y(T^*),\,p_L(T^*) \hsmm= \hsmm \mu_L(T^*) (x(T^*)\hsmm -\hsmm y(T^*))\}.
\end{eqnarray*}
\end{proposition}
\vspace{-.3cm}
Again, to unburden the presentation, the proof of this result is included in the Appendix.

\vspace{-.5cm}
\subsection{Flattening the Bilevel Approximation Problem and Application of Ekeland's Variational Principle}\label{subsec:flat}
Now, we construct an approximate flattened single-level problem to which the Maximum Principle of Pontryagin can be applied. By considering the change of time variable, and the state variable component $z$, $\dot z=|u_L|^2$, $z(0)=0$, we state the flattened single-level approximated problem as an equivalent optimal control problem $(P_F^k)$ on a fixed time interval $[0,T^*]$, (once again, we do not relabel),

\vspace{-.4cm}
$$\mbox{Minimize }t(T^*)\mbox{ subject to } (D_F^k)\left\{\begin{array}{l}
\dot y \hsmm =\hsmm v\omega, \;\dot x \hsmm =\hsmm f^k(y,x,u,u_0)\omega, \; \dot z  \hsmm =\hsmm |u_L|^2\omega, \; \dot t \hsmm = \hsmm \omega \; [0,T^*]\hsmm -\hsmm \mbox{a.e.}, \nonumber\\
y(0) \hsmm =\hsmm y_0,  \; x(0) \hsmm\in\hsmm Q_1\hsmm+\hsmm y_0,\;z(0) \hsmm =\hsmm 0, \; t(0) \hsmm =\hsmm 0,\; y(T^*) \hsmm\in\hsmm \bar E, \nonumber \\
v  \hsmm\in\hsmm  \mathcal{V}, \; u \hsmm\in\hsmm \mathcal{U}, \; u_0 \hsmm\in\hsmm  \mathcal{U}_0, \; \omega \hsmm\in\hsmm L_2([0,T^*];\R^+),\nonumber \\
h_H(y) \leq 0 , \;\; h_L(x,y) \leq 0  \;\; \forall t \in [0,T^*], \quad z(T^*) - \varphi^{k}(\omega,v) \leq 0, \label{ineq-varphi}\end{array}\right.
$$
where $\varphi^k(\omega,v)$ is as defined in (\ref{expression of value function}). The solution to this problem is not known, and, thus, in order to pursue with the derivation of the necessary conditions of optimality, we shall use Ekeland's variational principle, (cf. \cite{ekeland1974variational}, \cite[Theorem 3.3.1]{vinter2010optimal}).  Since the approximating family of problems is only at the lower level, the extension of Propositions \ref{proposition: nonempty feasible set}, and \ref{prop: uniqueness solution} to the current flattened control problem required to construct an appropriate sequence of auxiliary optimal control problems is straightforward. By ``appropriate'' it is meant that the assumptions underlying the application of Ekeland's variational principle hold.
\vspace{-.2cm}
\begin{proposition}\label{propo: ekeland principle} Let H1-H6 be in force and let $(y^*,x^*,z^*,t^*,v^*,u^*,u_0^*,\omega^*)$, with $t^*(T^*)=T^*$,  be an optimal process to problem $(P_F)$. Then, for some positive sequence $\{\varepsilon_k\}$, with $\varepsilon_k \downarrow 0$, as $k\to\infty$ there exists, for each $k$, a solution $(y_k^*, x_k^*, z_k^*,t_k^*,v_k^*, u_k^*,u_{0,k}^*,\omega_k^*)$ to the following perturbed version of $(P_F^k)$
\vsm\vsm
\begin{eqnarray*}\hsm(\widetilde P_F^k)&&\mbox{ Minimize }  t_k(T^*)\hsmm +\hsmm \varepsilon_k \left( | x_k^*(0)\hsmm -\hsmm x_k(0)|\hsmm+ \hsmm \Delta((v_k^*,u_k^*,u_{0,k}^*,\omega_k^*), (v_k,u_k,u_{0,k},\omega_k))\right) \nonumber \\
&& \hsm\hsm \mbox{ subject to } (D_F^k),\mbox{ and }  | x^* (0)\hsmm -\hsmm x_k(0)|\hsmm +\hsmm \Delta((v^*,u^*,u_{0}^*,\omega^*), (v_k,u_k,u_{0,k},\omega_k))\hsmm < \hsmm\varepsilon_k, \nonumber\end{eqnarray*}
\vspace{-.6cm}

\noindent where, $\Delta(\cdot.\cdot)\hsmm :\hsmm (\mathcal{V}\hsmm\times\hsmm \mathcal{U} \hsmm\times\hsmm \mathcal{U}_0\hsmm\times\hsmm L_2)\times (\mathcal{V}\hsmm\times\hsmm \mathcal{U}\hsmm\times\hsmm\mathcal{U}_0\hsmm\times\hsmm L_2)\to [0,\infty)$ is an extension of (\ref{expression of Ekeland metric}) to encompass the additional control $\omega$ in $L_2$. Thus, for a subsequence (we do not relabel), we have
\vsm\vsm
\begin{equation} (y_k^*, x_k^*,z_k^*,t_k^*) \to ( y^*, x^*, z^*, t^*) \textrm{ uniformly, and} \;\;(v_k^*,u_k^*,u_{0_k}^*,\omega_k^*) \to ( v^*, u^*, u_0^*, \omega^*) \textrm{ a.e. on } [0,T^*]. \nonumber\end{equation}
\end{proposition}
\vsm\vsm\vsm
A detailed proof of this proposition can be found in the Appendix.

\vspace{-.3cm}
\subsection{Partial Calmness of the Approximate Bilevel Problem $(\widetilde P_F^k)$}\label{subsec:partial}
From Proposition \ref{subgrad-value}, the value function to the lower level problem, $\varphi^k$, is Lipschitz continuous. This Lipschitzianity is a complementary property crucial for handling the degeneracy resulting from the functional inequality constraint $z(T^*)\hsmm - \hsmm \varphi^{k}(\omega,v) \leq 0$ in $(D^k_F)$, as it will be shown next. This degeneracy is due to the fact that, in the presence of this inequality, the standard constraint qualifications, such as Mangasarian-Fromovitz or linear independence constraint, do not hold.

In this stage, we eliminate the pathological situation of the degeneracy caused by the functional inequality constraint $z(T^*)\hsmm - \hsmm \varphi^{k}(\omega,v) \leq 0$ in $(D^k_F)$, by imposing the so-called {\it partial calmness condition} on $(P_F)$, and we use an approach based on an exact penalization technique. This consists in replacing the original problem by an equivalent one (i.e., with same solution) where the inequality $z(T^*)\hsmm -\hsmm \varphi^{k}(\omega,v) \leq 0$ is eliminated as a constraint, and a related penalty term is added to the cost functional. This idea has been used in \cite{ye1995optimality,ye1997optimal} to obtain necessary conditions of optimality, and later in \cite{dempe2014necessary,dempe2007new,benita2016bilevel,benita2016bilevelstateconstraint} in a wide variety of instances. In what follows, we give results on the partial calmness of the approximate problem $(\widetilde P_F^k)$, by using the notation of the previous subsection.
\begin{proposition}\label{Prop-part-cal}
The partial calmness of $(P_F)$ implies the partial calmness of $(\widetilde P_F^k)$ for any positive sequence $\{\varepsilon_k\} $ with $ \varepsilon_k\hsm \downarrow\hsm 0$.
\end{proposition}
\begin{proposition}\label{theorem:partial calmness equivalency}
Let  $ (\chi_k^*,\bar\theta_k^*) $ be a solution to $(\widetilde P_F^k)$. If $(\widetilde P_F^{k})$ is partially calm at $(\chi_k^*,\bar\theta_k^*)$ with modulus $\rho_k$, then $(\chi_k^*,\bar\theta_k^*)$ solves the auxiliary optimal control problem
\begin{eqnarray*}({\bf P}_F^k)\mbox{ Minimize }&&  t_k(T^*) \hsmm +\hsmm \varepsilon_k \left(|x_k(0)\hsmm -\hsmm x_k^*(0)| \hsmm +\hsmm
\Delta (\bar\theta_k,\bar\theta_k^* )\right)\hsmm + \hsmm \rho_k \left(z_k(T^*)-\varphi_k(\omega_k, v_k)\right),\nonumber\\
\mbox{subject to } && (D_F^k), \mbox{ and } |x_k(0)\hsmm -\hsmm x_k^*(0)| \hsmm +\hsmm \Delta (\bar\theta_k,\bar\theta^* )\hsmm < \hsmm\varepsilon_k,
\end{eqnarray*}
where $\bar\theta_k^*\hsmm:=\hsmm(v_k^*,u_k^*,u_{0,k}^*,\omega_k^*), \;\bar\theta^*\hsmm:=\hsmm(v^*,u^*,u_{0}^*,\omega^*),\; \bar\theta_k\hsmm :=\hsmm(v_k,u_k,u_{0,k},\omega_k)$, and $\Delta(\cdot,\cdot)$ is as in (\ref{expression of Ekeland metric}) expanded to encompass the additional control.
\end{proposition}
\begin{proof}
The proof goes along the lines of \cite[Proposition 3.4]{benita2016bilevelstateconstraint}, by directly using the definition of partial calmness (cf. Definition \ref{def:partial calmness}), the Lipschitz continuity of $\varphi_k$, and by noting that $z_k^*(T^*) \hsmm -\hsmm \varphi^k( \omega_k^*, v_k^*)=0$.
\end{proof}
\subsection{Necessary Conditions of Optimality for $({\bf P}_F^k)$}\label{subsec:necessary condition approx}

We first note that, from the Lipschitz continuity of $\varphi^k$, as discussed in the previous step, we can write
$\partial^C \varphi^k(\omega,v)= \textrm{co } \partial \varphi^k(\omega,v)  $ where $ \partial^C \varphi^k(\omega,v)$ is the Clarke subdifferential of $\varphi^k$ at $(\omega,v)$, and $\textrm{co } \partial \varphi^k(\omega,v)$ stands for the closure of the convex hull of the Mordukhovich (limiting) subdifferential of $\varphi^k$ at $(\omega,v)$. Moreover the Clarke subdifferential is homogeneous in the sense that
$\partial^C (-\varphi^k) (\omega,v) = - \partial^C \varphi^k(\omega,v)$ while this property fails for the Mordukhovich (limiting) subdifferential. This property will be crucial in the derivation of necessary conditions of optimality presented below.

Now, we are ready to establish the necessary conditions of optimality to $({\bf P}_F^k)$. We adopt the approach used in the proof of Theorem 3.2 in \cite{ye1997optimal}. Clearly, the cost functional of $({\bf P}_F^k)$ makes it a nonstandard optimal control problem. Thus, in order to derive the necessary conditions of optimality, the problem is seen as a nonlinear programming problem in appropriate infinite dimensional spaces. Then, we apply the Fermat principle to the associated Lagrangian (the required assumptions hold) and decode the obtained conditions in terms of the data of $({\bf P}_F^k)$. Albeit rather technical, these steps are straightforward, and, thus, we omit them, and jump directly to the resulting conditions which we choose to cast in the Gamkrelidze's form of the Hamilton-Pontryagin function.

Let $c(\gamma_k,y,x)$ be as defined in Subsection \ref{subsec:approx lower level}, and denote $ l(a,b)$ by either $ l^*_k$ or by $l^*_k(a)$ if its arguments are, respectively, either $(a_k^*,b_k^*)$ or $(a,b^*_k)$. Recall also $f^k(y,x,u,u_0)\hsm=\hsm f(x,u)- u_0 c(\gamma_k,y,x)(x-y)$, and let $ r^k =\lambda^k \rho^k$. The Hamilton-Pontryagin function in Gamkrelidze's form of $({\bf P}_F^k)$ is given by
\vspace{-.3cm}
\begin{equation}
\vspace{-.3cm}H_F^k(\chi,\bar\theta, q_H^k,q_L^k,\nu_H^k,\nu_L^k, \lambda^k,r^k; \bar \theta_k^*)\hsmm =\hsmm \widetilde H^k_F(y,x,v,u,u_0,q_H^k,q_L^k,\nu_H^k,\nu_L^k, \lambda^k,r^k)\omega\hsmm -\hsmm\lambda^k \varepsilon_k|\bar\theta\hsmm -\hsmm\bar\theta_k^*|,\nonumber
\end{equation}
\vspace{-.3cm}
\noindent where $\chi=(y,x,z,t)$, $\bar\theta=(v,u,u_0,\omega)$, and\vsm
\begin{equation}\tilde H^k_F(\cdot)\hsmm =\hsmm\langle q_H^k\hsmm - \hsmm \nu_H^k(y\hsmm -\hsmm q_0)\hsmm +\hsmm \nu_L^k(x\hsmm -\hsmm y), v\rangle \hsmm+\hsmm \langle q_L^k\hsm - \hsmm \nu_L^k(x\hsmm -\hsmm y), f(x,u)\hsmm -\hsmm c(\gamma_k,y,x)(x \hsmm -\hsmm y)u_0\rangle
\hsmm - \hsmm r^k(|u|^2\hsm+\hsmm |u_0|^2)\hsmm -\hsm \lambda^k \hsm .\label{til-H}
\end{equation}
Let $ (\chi_k^*,\bar\theta_k^*)$ be a minimizer to $({\bf P}_F^k)$. Then, there exists a multiplier\vsm
$$( q_H^k,q_L^k,\nu_H^k,\nu_L^k,\lambda^k,r^k) \hsm\in\hsm AC\hsm \times\hsm AC\hsm\times\hsm N\hspace{-.03cm}BV \hsm\times\hsm N\hspace{-.03cm}BV\hsm\times\hsm\R^+ \hsm\times\hsm \R^+ $$
\noindent not all zero, where $\nu^k_H,\nu^k_L$ are non-increasing, being $\nu^k_H$, and $\nu^k_L$ constant on the subsets of $[0,T^*]$ in which the inequalities $|y^*_k-\tilde y|\hsmm >\hsmm R_1$, $ \forall \tilde y\hsmm\in \hsmm\partial Q$, and $ |y_k^*\hsmm -\hsmm x_k^*|\hsmm <\hsmm R_1$ respectively hold, satisfying:
\begin{itemize}
\item[1.] Adjoint equations.
\vspace{-.2cm}
\begin{eqnarray}
-\dot q_H^k(t) &\in & (- \nu_H^k(t)v_k^*(t)  + \partial_y \langle q_L^k(t)- \nu_L^k(t)( x^*(t)_k- y^*_k(t)), f^{k,*}(t)\rangle- \nu^k_L (t)v^*_k(t) )\omega^*_k(t),\nonumber\\
-\dot q_L^k(t) &\in & ( \partial_x\langle q_L^k(t)-\nu_L^k(t)(x^*_k(t)- y^*_k(t)),  f^{k,*}(t)\rangle+ \nu^k_L (t) v^*_k(t)) \omega^*_k(t), \qquad [0,T^*]-\mbox{a.e.}. \nonumber
\end{eqnarray}
\vspace{-1cm}
\item[2.] Boundary conditions.\vspace{.1cm}
\begin{itemize}
\item[] $ q_H^k(0)\hsmm\in\hsmm\R^n$, $ q_H^k(T^*) \hsmm\in\hsmm -N_{\bar E}(y^*_k(T^*))+ \nu^k_H(T^*)(y^*_k(T^*) \hsmm -\hsmm q_0)\hsmm -\hsmm \nu_L^k(T^*) (x^*_k(T^*) \hsmm -\hsmm y^*_k(T^*))$,\vspace{.1cm}
\item[] $q_L^k(0)\hsmm\in\hsmm N_{Q_1+y_0}( x^*_k(0) )\hsmm+\hsmm \nu_L^k(0) (x^*_k(0)\hsmm -\hsmm y^*_k(0))\hsmm -\hsmm \lambda^k {\varepsilon_k}\xi_{x,k}$, $q_L^k(T^*) \hsmm =\hsmm \nu_L^k(T^*)(x^*_k(T^*) \hsmm -\hsmm y^*_k(T^*))\vspace{.1cm}.$
\end{itemize}
\item[3.] Maximum condition. $(u_k^*,v_k^*,\omega_k^*,u^*_{0,k})$ takes values on $U\hsm\times\hsm V\hsm\times\hsm \R_0^+\hsm\times\hsm [0,1]$, maximizing $[0,T^*]$-a.e., the mapping
\begin{eqnarray} (u,v,\omega,u_0) & \mapsto & H_F^k(\chi^*,u,v,\omega,u_0, q_H^k,q_L^k,\nu_H^k,\nu_L^k, \lambda^k,r^k; u_k^*,v_k^*,\omega_k^*,u^*_{0,k})\hsmm +\hsmm r^k \varphi^k(w,v)\label{cond:max all}\end{eqnarray}
\noindent along $ (\chi^*(t),q_H^k(t),q_L^k(t),\nu_H^k(t),\nu_L^k(t))$.
\end{itemize}
\noindent Given the structure of $H_F^k$, and the fact that $\omega \geq 0 $, the maximum condition (\ref{cond:max all}) amounts to:
\begin{itemize}
\item For almost all $t\hsmm\in\hsmm [0,T^*]$, $u_k^*$ maximizes on $U$, and $u_{0,k}^* $ maximizes on $[0,1]$ the mappings
\begin{equation}u \mapsto \left(\langle q_L^k(t)\hsmm - \hsmm \nu_L^k(t) (x^*_k(t)\hsmm - \hsmm y^*_k(t)),f(x^*_k(t),u)\rangle\hsmm  -\hsmm r^k |u|^2 \right)\omega_k^*(t)\hsmm -  \hsmm \lambda^k\varepsilon_k|u\hsmm -\hsmm u^*_k(t)|, \label{max-uk}\end{equation}
\begin{equation}  \mbox {and }\qquad u_0  \mapsto \left( \bar m_k(t) u_0\hsmm -\hsmm r^k|u_0|^2\right)\omega_k^*(t)\hsmm - \hsmm \lambda^k\varepsilon_k|u_0-u_{0,k}^*(t)|,\label{max-u0}\end{equation}
\noindent respectively, where $\bar m_k(t)\hsmm =\hsmm - c(\gamma_k, y^*_k(t),x^*_k(t))\langle q_L^k(t) \hsmm -\hsmm \nu_L^k(t) (x^*_k(t) \hsmm-\hsmm y^*_k(t)), x^*_k(t) \hsmm -\hsmm y^*_k(t)\rangle$.
\item For almost all $t\hsmm\in\hsmm [0,T^*]$, the pair $(\omega_k^*,v_k^*)$ satisfies
\begin{equation} - ( \tilde H_F^{k,*}(t),\nabla_v\tilde H_F^{k,*}(t)\omega_k^*(t))\hsmm + \hsmm \lambda \varepsilon_k(\xi_{\omega,k}^*(t),\xi_{v,k}^*(t))\hsmm \in\hsmm r^k \partial_{(\omega ,v)}^C \varphi^k(\omega^*_k,v^*_k)(t)\hsmm + \hsmm N_{L_2}(\omega_k^*)(t)\hsmm\times\hsmm N_{\cal V}(v_k^*)(t). \label{max-vk+wk}\end{equation}
\end{itemize}
\noindent Here, $\xi_{a,k}^*(t)\hsm \in \hsm \partial|a \hsmm -\hsmm a^*_k(t)|_{a= a^*_k(t)}$. Note that $|\xi_{a,k}^*|\hsm =\hsm 1$, being $\partial^C_{(w,v)} \varphi^k (\omega^*_k, v^*_k)$ defined in Proposition \ref{subgrad-value} in terms of the corresponding set of multipliers $(p_H^k,p_L^k,\mu_H^k,\mu_L^k,\bar\lambda^k)$.

\subsection{Limit Taking}\label{subsec:limit taking}
In this last step, we show how our main result, Theorem \ref{theorem:main thm}, follows from the necessary conditions of optimality for $({\bf P}_F^k)$ by taking the appropriate limits as $k\hsm\to\hsm\infty$. The proof uses also Proposition \ref{propo: ekeland principle}, which asserts the existence of a subsequence $(\chi^*_k,\bar \theta_k^*)$ converging in an appropriate sense to $( y^*, x^*, z^*, t^*, v^*, u^*, u_0^*, \omega^*)$, with $ \dis \int_0^{T^*}\hsm\hsm\hsm\omega^*(\tau)d\tau\hsm =\hsm T^*$, the solution to the original bilevel problem $(P_F)$.

From the boundary conditions, the sequence of functions $\{\hsmm( q_H^k(T^*\hsmm), q_L^k(T^*\hsmm))\hsmm\}$ is uniformly bounded. Then, by Gronwall's Lemma, the $q_L^k$'s and $q_H^k$'s are uniformly bounded and $\dot q_H^k$'s, $\dot q_L^k$s are uniformly bounded in $L_1$. Thus, along a subsequence (we do not relabel), and by Dunford-Pettis theorem, $\{(q^k_H, q^k_L)\}$ converges uniformly to $(q_H,q_L)$, and $\{(\dot q^k_H, \dot q^k_L)\}$ converges weakly to $(\dot q_H,\dot q_L)$ in $L_1$ for some $q_H$, and $q_L $ in $ AC([0,T^*];\R^{n})$. Moreover, since $\{(\nu^k_H,\nu^k_L)\}$ and $\{(\lambda^k,r^k)\}$ are uniformly bounded, then, again, by standard arguments, there exist subsequences such that $\{(\nu^k_H,\nu^k_L)\}$ converges point-wisely to $(\nu_H,\nu_L)$ and $\{(\lambda^k,r^k)\}$ to $(\lambda ,r)$, for some nonincreasing functions $ \nu_L$, and $\nu_H$, and numbers $\lambda \hsm\geq\hsm 0$, and $ r \hsm\geq\hsm 0$. Furthermore, $\nu_H$, and $\nu_L$ are constant on $\{t\hsm\in\hsm[0,T^*]\hsm:\hsm |y^*\hsm -\hsmm \tilde y| \hsm >\hsm R_1\, \forall\,\tilde y\hsm\in \hsm\partial Q\}$, and on $\{t\hsmm\in\hsmm [0,T^*]\hsm:\hsm |y^*\hsmm -\hsmm x^*|\hsmm <\hsmm R_1\}$, respectively. Note also that, by taking the limit as $k\hsm\to\hsm\infty$ in (\ref{max-u0}), and using a result in the proof of Proposition \ref{prop: uniqueness solution}, we obtain the function $\sigma$ defined in section \ref{sec: main theorem}. By the compactness of trajectories theorem \cite[Theorem 2.5.3]{vinter2010optimal}, and Proposition \ref{propo: ekeland principle}, the pair $(q_H,q_L)$ satisfies the respective adjoint equations and boundary conditions of our main result.  Moreover, by passing to the limit in (\ref{max-uk}), and (\ref{max-u0}) leads to (\ref{max-u}).

Let us obtain the remaining conditions: conservation law (\ref{time-transversality}), and maximum condition on $v^*$ (\ref{max-vw}). First, let us consider (\ref{grad-aval-k}) in Proposition \ref{subgrad-value}. Observe that, due to the time independence of the data of $(P_L^k)$, there exists constant $\bar c_k$ such that $ \overline H_L^{k*}(t) \hsmm =\hsmm \bar c_k$. Moreover, the normality of $(P_L^k)$ follows easily from its structure, and the assumptions on its data, and thus, $\bar \lambda^k $ in (\ref{grad-aval-k}) is strictly positive. Thus, we have $ \partial^C_{(w,v)} \varphi^k (\omega^*_k, v^*_k)\hsmm =\hsmm \{ c_k\}\hsmm\times\hsmm\partial^C_{v} \varphi^k (\omega^*_k, v^*_k) $, where $\dis c_k=\frac{\bar c_k}{\bar \lambda^k}$. By rescaling, $ c_k$ takes values in a given compact set for all $k$, and, by choosing an appropriate subsequence, $ \{c_k\}$ converges to a certain constant $ c$.

Now, by taking the limits in (\ref{max-vk+wk}), with further subsequence extraction, and, given that $\partial^C \varphi^k (\omega^*_k, v^*_k)$ converges (in the Kuratowski sense) to $\partial^C \varphi (\omega^*, v^*)$,  we obtain $\dis H_H^*(t)\hsmm -\hsmm \lambda\hsmm\in\hsmm r c\hsmm-\hsmm N_{L_2([0,T^*],\R_0^+)}(\omega^*)$ and $ (-q_H\hsmm +\hsmm \nu_H(y^*\hsmm -\hsmm q_0) \hsmm - \hsmm \nu_L(x^*\hsmm - \hsmm y^*))\omega^*(t)\hsm\in\hsm -r\partial^C_v\varphi(\omega^*,v^*))\hsmm +\hsmm N_{\cal V}( v^*)$. By considering the original time parametrization (we do not relabel), we conclude that the last inclusion becomes (\ref{max-vw}) where the arguments of $\varphi $ now are $(T^*,v^*)$. By noting the existence of solutions to $(P_L)$ in a neighborhood of $T^*$, and that $\varphi$ is Lipschiz continuous, we may choose $ c\hsmm\in\hsmm\partial^C_T\varphi(T^*,v^*)$ in the original parametrization. Note also that, in the reparametrized time, $N_{L_2([0,T^*],\R_0^+)}(\omega^*)\hsmm =\hsmm\{0\}$ a.e.. With these two observations, we conclude that $H^*_H(t)\hsmm =\hsmm \lambda\hsmm + \hsmm r c $ and thus, the conservation law (\ref{time-transversality}) holds. This concludes the proof of Theorem \ref{theorem:main thm}.

\section{Conclusions}\label{sec: conclusion}
This article deals with an instance of a bilevel optimal control problem with the dynamics at the lower level given by a sweeping process. We formulate and prove a theorem on the necessary conditions of optimality for a simple instance of this problem, i.e., the cost functional of the lower level is scalar valued. To the best of our
knowledge, this is the first time that a problem of this type is addressed. The main techniques used in the main theorem's proof are the approximation of the sweeping
process at the lower level by a smooth Lipschitz function in the state variable and the flattening of the bilevel structure to a single-level problem with an objective function involving a penalty term representing the value function of the lower level problem.
	
There are several possible avenues to extend this work. One might investigate problems with multi-objective functions at the lower level. In these circumstances, a new solution concept is required, for instance in the sense of Pareto-like solution, among other possibilities. Moreover, one might consider problems with nonsmooth data and more complicated state constraint sets, leading to several challenges in handling the passing to the limit. These open questions will be the subject of forthcoming research.

\begin{acknowledgements}
The authors acknowledge the partial support of FCT R\&D Unit SYSTEC - POCI-01-0145-FEDER-006933, and R\& D Projects STRIDE - NORTE-01-0145-FEDER-000033, MAGIC - PTDC/POCI-01-0145-FEDER-032485, SNAP - NORTE-01-0145-FEDER-000085, HARMONY - NORTE-01-0145-FEDER-031411 funded by ERDF | COMPETE2020 | FCT/MEC | PT2020 | NORTE 2020, POCI and PIDDAC.
\end{acknowledgements}

\appendix  
\section*{Appendix}
We provide here the proofs of the propositions \ref{prop: uniqueness solution}-\ref{Prop-part-cal} stated in Section \ref{sec:proof}.

{\bf Proof of Proposition \ref{prop: uniqueness solution}.}  Fix some $k\in\N$. From Proposition \ref{proposition: nonempty feasible set}, let $(T_k,y_k)$ be a pair such that ${\mathcal F}^k_L(T_k,y_k) \neq \emptyset$, with $y_k(0)=y_0$, and such that $\dot y_k=v_k\in {\cal V}$. Then, there exists a triple $(x_k,u_k,u_{0,k})\in AC([0,T_k];\R^n)\hsmm\times\hsmm {\cal U}\hsmm \times\hsmm{\cal U}_0$, satisfying $(D^k)$. Here, and in what follows, the sets  ${\cal U}$ and ${\cal U}_0$ are defined on the time interval $[0,T_k]$.
	
The non-emptiness of the set of sequences $\left\{(T_k,y_k, x_k,v_k, u_k,u_{0,k})\right\}_{k=1}^\infty$ satisfying $1)-3)$ follows from the fact that, for each $k$, the velocity set of the dynamics of $(D^k)$ contains the one of the dynamics of $(P_L)$ (as observed in the proof of Proposition \ref{proposition: nonempty feasible set}), and, thus, obvious fact that the constant sequence $\{(T,y,x,v,u,u_0)\}$, which is feasible for $(D^k)$, satisfies $(x,u,u_0)\hsm\in\hsm{\cal F}_L(T,y)$. By inspection of the data of $(P_L)$ and $(P_H)$, we can construct a sequence with the stated properties, and such that, on $[0,\max\{T_k, T_{k+1}\}]$, and for some constant $\widetilde M\hsm\in\hsmm \R^+\hsmm$, $\Delta\left((v_k,u_k,u_{0,k}), (v_{k+1},u_{k+1},u_{0,k+1})\right) \hsm\leq\hsm \widetilde M 2^{-k}.$

The boundedness of $f(\cdot,\cdot)$ (by assumption H1), and of the approximation to the truncated normal cone defined by (\ref{approx of normal cone}) implies that $\{\dot x_k\}$ is a bounded sequence in $L_1$. Let us see now that $\{x_k\}$ is an equicontinous and uniformly bounded sequence. Indeed, take any $\epsilon>0$, and consider $\tau_1$, and $\tau_2$ such that $T_k>\tau_2 > \tau_1$, and $\dis \tau_2-\tau_1\hsm \leq\hsm \frac{\epsilon}{M_1+M}$. Then,

\vspace{-.6cm}
\begin{eqnarray*}
|x_k(\tau_2)\hsm -\hsm x_k(\tau_1)| \hsmm \leq\hsm  \int_{\tau_1}^{\tau_2}\hsm\hsm\hsmm |\dot x_k(s)| ds
& \leq & \hsm\int_{\tau_1}^{\tau_2}\hsm\hsm\hsm\left( |f(x_k(s),u_k(s))| \hsmm+\hsmm |u_{0,k}(s) c(\gamma_k, x_k(s),y_k(s)) (x_k(s)\hsmm -\hsmm y_k(s))|\right) ds \hsmm\leq \hsmm(M_1\hsmm+\hsmm M) (\tau_2\hsmm-\hsmm \tau_1)\hsm \leq\hsm \epsilon.
\end{eqnarray*}

\vspace{-.3cm}
Since the $x_k(0)$'s are bounded, the uniform boundedness of $\{x_k\}$ is obtained from the fact that, for all $t\hsmm\in\hsmm [0,\overline T]$, where $\dis \overline T= \max_{i \in\N}\{T,T_i\}$, we have $ |x_k(t)|\hsmm \le\hsmm |x_k(0)|\hsmm +\hsmm |x_k(t)-x_k(0)| \hsmm\leq\hsmm |x_k(0)| \hsmm +\hsmm (M_1\hsm +\hsm M) t \hsmm\leq\hsmm | x_k(0)| \hsmm +\hsmm (M_1\hsm +\hsm M) \overline T.$
Due to the uniform boundedness, and the equicontinuity of $\{x_k\}$, we conclude, from the application of Arzela-Ascoli theorem, the existence of a subsequence (without relabeling) such that $x_k$ converges uniformly to an absolutely continuous function $x$. From the boundedness of $\{\dot x_k\}$ in $L_1$, and, from the Dunford-Pettis theorem, there exists a subsequence (we do not relabel) $\{\dot x_k\}$ converging weakly in $L_1$ to some function $\xi\in L_1$. Moreover, since $\{x_k(0)\}$ is a bounded sequence, we arrange by further subsequence extraction that $ x_k(0) \to x(0)$ for some $x(0)\in Q_1+y_0$. Hence, $\dot x(t) = \xi(t)$ a.e., and

\vspace{-.5cm}
$$  x_k(t) = x_k(0) + \int_{0}^{t} \dot x_k(s) ds \to_k x(t) := x(0) + \int_{0}^{t} \xi(s) ds  \qquad \mbox{uniformly},$$

\vspace{-.2cm}
Next, we prove that $\dot x(t) \in f(x(t),u(t))- N_{Q_1+y(t)}^M(x(t))$ a.e.. From the convexity of $Q_1$, and from the definition of the normal cone to a convex set, this amounts to show the following inequality

\vspace{-.6cm}
\begin{equation} \label{inequality normal cone belonging}  \langle f(x(t),u(t))- \dot x(t), \tilde x- x(t) \rangle \leq 0 \quad \forall \tilde x \in Q_1+y(t).  \end{equation}

\vspace{-.2cm}
For this purpose, let us examine the quantity $ \langle f(x_k(t),u_k(t))- \dot x_k(t), \tilde x- x_k(t) \rangle  $ for the control $u_k$, and the pair $(T_k,y_k)$, for which we recall that $\mathcal{F}_L^k(T_k,y_k)\neq \emptyset$. We consider two cases: $\tilde x \in \mbox{int} (Q_1\hsmm +\hsmm y_k(t))$, and $\tilde x \in \partial(Q_1\hsmm +\hsmm y_k(t))$.
	
We start by taking an arbitrary $\tilde x \in \mbox{int} (Q_1\hsmm +\hsmm y_k(t))$. From the a.e. equality

\vspace{-.5cm}
$$ f(x_k(t),u_k(t))- \dot x_k(t)  = u_{0,k}(t) c(\gamma_k, x_k(t),y_k(t))( x_k(t)-y_k(t)), $$

\vspace{-.3cm}
\noi the convexity of $h_L$, and the expression of its gradient, $\nabla_x h_L(x_k(t),y_k(t))= x_k(t)-y_k(t)$, (there is no loss of generality in assuming $x_k(t)\neq y_k(t)$), we conclude that

\vspace{-.5cm}
\begin{equation} \langle f(x_k(t),u_k(t))- \dot x_k(t), \tilde x- x_k(t) \rangle \leq   u_{0,k}(t) c( \gamma_k,x_k(t),y_k(t))(h_L(\tilde x,y_k(t))- h_L(x_k(t),y_k(t))) .\label{inequality}\end{equation}

\vspace{-.2cm}
\noi Define $ Z_k(t,\tilde x):= u_{0,k}(t) c( \gamma_k,x_k(t),y_k(t))(h_L(\tilde x,y_k(t)) - h_L(x_k(t),y_k(t)))$.
From (\ref{eq:condition on gamma_k}), we have  that, either $  \dis \lim_{k\to\infty}\hsmm Z_k(t,\tilde x) = 0$, if $ x_k(t)\hsmm \in\hsmm \mbox{int} (Q_1\hsmm +\hsmm y_k(t))$,  or $ Z_k(t,\tilde x)\hsmm \leq\hsmm 0 $ if $ x_k(t) \hsmm \in\hsmm \partial(Q_1\hsmm +\hsmm y_k(t))$. Let us take the limit in (\ref{inequality}). As already proved, the sequence $\{\dot x_k\}$ converges weakly in $L_1$ to some $\xi$, such that $\dot x(t)=\xi(t)$ a.e.. By Mazur's theorem, there exists a sequence $\{\chi_n\}$, such that $\dis\chi_n= \sum_{k=n}^{\infty} \beta_k\dot x_k$, where $\beta_k \hsm\geq\hsm 0$ and $\dis\sum_{k=n}^{+\infty} \beta_k=1$, and $\chi_n$ converges in the $L_1$-norm to $\xi$. This also implies that, up to a subsequence (we do not relabel), $\{\chi_n\}$ converges to $\xi$ a.e..
	
Now, let us take a Lebesgue point $t\hsmm\in\hsmm [0,T_k]$, and some $\tau > 0$ such that $t\hsmm +\hsmm \tau \leq T _k$. Then, we may write

\vspace{-.6cm}
\begin{eqnarray}
\frac{1}{\tau}  \sum_{k=n}^{\infty} \beta_k \int_{t}^{t+\tau} \hsm\hsm\hsm Z_k(s,\tilde x) ds  &\geq &	\frac{1}{\tau} \sum_{k=n}^{\infty}\beta_k \int_{t}^{t+\tau} \hsm\hsm\hsm \langle f(x_k(s),u_k(s))- \dot x_k(s), \tilde x- x(s) \rangle ds \label{relation mazur}\\
&& \hspace{3cm} + \frac{1}{\tau} \sum_{k=n}^{\infty}\beta_k \int_{t}^{t+\tau} \hsm\hsm\hsm \langle f(x_k(s),u_k(s))- \dot x_k(s), x(s)- x_k(s) \rangle ds \nonumber\\
& \geq &	\frac{1}{\tau} \sum_{k=n}^{\infty}\beta_k \int_{t}^{t+\tau} \hsm\hsm\hsm \langle f(x_k(s),u_k(s))- \dot x_k(s), \tilde x- x(s) \rangle ds - M\frac{1}{\tau}\sum_{k=n}^{\infty}\beta_k \int_{t}^{t+\tau} \hsm\hsm\hsm | x(s)- x_k(s) | ds, \label{ineq-a}
\end{eqnarray}

\vspace{-.3cm}
\noi being inequality (\ref{ineq-a}) due to the boundedness of $|f(x_k(s),u_k(s))- \dot x_k(s)|$. From the uniform convergence of $x_k$ to $x$, we have that $\dis M\frac{1}{\tau}\int_{t}^{t+\tau} \hsm\hsm\hsm | x(s)- x_k(s) | ds \to 0$. Moreover, we can write

\vspace{-.5cm}
\begin{eqnarray*}  && \frac{1}{\tau} \sum_{k=n}^{\infty}\beta_k \int_{t}^{t+\tau} \hspace{-.3cm}\langle f(x_k(s),u_k(s))- \dot x_k(s), \tilde x- x(s) \rangle ds  = \\ && \hspace{2cm}\frac{1}{\tau}  \int_{t}^{t+\tau} \hspace{-.3cm}\langle f(x(s),u(s)) -\xi(s), \tilde x- x(s) \rangle ds + \frac{1}{\tau}  \int_{t}^{t+\tau} \hsm\hsmm \left\langle \sum_{k=n}^{\infty}\beta_k [f(x_k(s),u_k(s))- f(x(s),u_k(s))], \tilde x\hsmm - \hsmm x(s) \hsmm\right\rangle ds\\
&& \hspace{2cm} + \frac{1}{\tau}  \int_{t}^{t+\tau}\hsm\hsmm \left\langle \xi(s)- \sum_{k=n}^{\infty}\beta_k \dot x_k(s), \tilde x\hsmm -\hsmm x(s) \hsmm\right\rangle ds + \frac{1}{\tau}  \int_{t}^{t+\tau} \hsm\hsmm \left\langle \sum_{k=n}^{\infty}\beta_k f(x(s),u_k(s))\hsmm -\hsmm f(x(s),u(s)), \tilde x\hsmm -\hsmm x(s)\hsmm \right\rangle ds.
\end{eqnarray*}

\vspace{-.3cm}
It is easy to check that each one of the last three terms converges to zero. The second term is due to the fact that $\dis \sum_{k=n}^{\infty}\beta_k \dot x_k(t) \to \xi(t)$, the third one follows from the Lipschitz continuity of $f$ w.r.t. $x$, and the uniform convergence of $x_k$ to $x$, and the last one is a consequence of Mazur's theorem, by taking into account that $\{u_k\}$ weakly converges in $L_1$ to $u$, which entails that, on $[t,t+\tau)$, when $k\to\infty$, $\dis \sum_{k=n}^{\infty}\beta_k f(x(s),u_k(s))$ converges to $ f(x(s),u(s))$. By passing the relation (\ref{relation mazur}) to the limit in $n$, we obtain $\dis \frac{1}{\tau}  \int_{t}^{t+\tau} \hspace{-.3cm}\langle f(x(s),u(s)) -\xi(s), \tilde x- x(s) \rangle ds  \le \lim_{n\to \infty}\frac{1}{\tau}  \sum_{k=n}^{\infty} \beta_k \int_{t}^{t+\tau} \hspace{-.3cm} Z_k(s,\tilde x) ds \leq 0. $
Now, by taking the limit $\tau \to 0$, and, since $t$ is a Lebesgue point, then, for all $\tilde x\hsmm \in\hsmm \mbox{int}(Q_1\hsmm +\hsmm y(t))$, we conclude that

\vspace{-.4cm}
\begin{equation}\label{inequality lebesgue point} \langle f(x(t),u(t)) -\dot x(t), \tilde x- x(t) \rangle  \leq 0. \end{equation}
	
\vspace{-.2cm}
Let us consider now, any point $\tilde x \hsmm\in\hsmm \partial(Q_1\hsmm +\hsmm y)$. Take some $\lambda\hsmm \in\hsmm [0,1)$. Clearly,
$   h_L(\lambda(\tilde x-y)+y,y)= \frac{1}{2}(\lambda^2 - 1)R_1^2 <0.  $
We are, therefore, reduced to the previous case replacing $\tilde x$ in (\ref{inequality lebesgue point}) by $y+ \lambda (\tilde x -y)$. By taking the limit $\lambda\to 1$, and by using the continuity of $h_L(\cdot,y)$, we have that, for all $t$, $h_L(x(t),y(t)) \leq 0$. Thus, we conclude that $x(\cdot)$ is a solution to $(D)$.
	
It remains to show the uniqueness of the solution to $(D)$. This is a consequence of the hypomonotonicity of the truncated normal cone to a convex set (cf. \cite{poliquin2000local}), and of the application of the Gronwall's inequality. Indeed, let $(T,y)$ be a given parameter, and take a measurable $u$ such that $u(t) \in U$, $[0,T]$-a.e. By contradiction, let $x_1(\cdot)$, and $x_2(\cdot)$ be distinct solutions to $(D)$ with $x_1(0)=x_2(0)$, and let $\tilde v(x_1,y)$, and $\tilde v(x_2,y)$ be the velocity components of $x_1$, and $x_2$, respectively, in $N_{Q_1+y}^M(x_1)$, and $N_{Q_1+y}^M(x_2)$. Thus, we may write $	\langle \dot x_1- \dot x_2, x_1-x_2 \rangle\hsmm  = \hsmm\langle f(x_1,u) - \tilde v(x_1,y)- f( x_2,u) + \tilde v(x_2,y), x_1-x_2 \rangle$.
The Lipschitz continuity of $f(\cdot,u)$ (assumption H1) implies that

\vspace{-.5cm}
\begin{equation}\label{equation1}
\langle \dot x_1- \dot x_2, x_1-x_2 \rangle \leq  K_f |x_1-x_2|^2 - \langle  \tilde v(x_1,y) - \tilde v(x_2,y), x_1-x_2 \rangle.
\end{equation}

\vspace{-.2cm}
From the hypomonotonicity of the truncated normal cone, we may write the inequality

\vspace{-.5cm}
\begin{equation}\label{equation2}
		\langle \tilde v(x_1,y) - \tilde v(x_2,y), x_1-x_2 \rangle \geq - |x_1-x_2|^2.
\end{equation}

\vspace{-.2cm}
By replacing (\ref{equation2}) in (\ref{equation1}), we obtain $\dis \frac{1}{2} \frac{d}{dt} (|x_1-x_2|^2) \leq  (K_f+1) |x_1-x_2|^2 $. By applying Gronwall's lemma, and using the fact that $x_1(0)=x_2(0)$, we obtain $x_1(\cdot)=x_2(\cdot)$, thus confirming the uniqueness of the solution to $(D)$.

\medskip
	
{\bf Proof of Proposition \ref{subgrad-value}.} Here, the index $k$ in the specification of the multipliers, the state, and control variables is omitted for the sake of simplicity. Let $(T,y)$ be a parameter inherited from $(P_H)$ such that $\mathcal{F}^k_L(T,y) \hsm\neq\hsm \emptyset$, and denote by $(x,u_L)$ a feasible control process to $(P_L^k(T,y))$ whose existence is asserted by Proposition \ref{proposition: nonempty feasible set}. Then, $\varphi^k(\omega,v)$ is finite where $(\omega, v)$ corresponds to $(T,y)$. Standard compactness arguments (see \cite[Lemma 2.4]{ye1997optimal}) allow us to conclude that, for a sequence $\{(T_j,y_j)\}$ converging in $L_2$ to $(T,y)$, and a sequence of $\{(x_j,u_j)\}$ with $(x_j,u_j)$ feasible for $(P^k_L(T_j,y_j))$, there is a subsequence $\{x_j\}$ converging uniformly to $x$ for which $\exists\, u_L$ feasible, and such that $(x,u_L)$ is feasible for $(P^k_L(T,y))$, and $\dis J_L(x(0),u_L; T,y)\hsm \leq\hsm \lim\inf_j J_L(x(0)_j,u_j; T_j,y_j)$. This, with the boundedness of $\varphi^k$ in its domain, entails its lower semicontinuity.

Moreover, let $(T,y)$ be such that $\partial^P\varphi^k (\omega,v) \neq \emptyset$. Assume that such $(x,u_L)$ is an optimal solution to $(P_L^k)$, i.e. $(x,u_L) \in \Psi^k(\omega,v)$, being the latter the set of optimal solutions to $(P_L^k)$, and that $ \partial^P \varphi^k(\omega,v)\hsm \neq \hsm\emptyset$. Remark that, by \cite[Lemma 2.4]{ye1997optimal} applied to $\varphi^k$, this property holds a.e.. Let $(\zeta_1,\zeta_2)\hsm \in\hsm \partial^P \varphi^k(\omega,v)$. Here, $(\omega,v)$ are the controls associated with $(T,y)$ as defined in (\ref{omega definition}). By definition of the proximal subdifferential in the Hilbert space $L_2$, for some $\bar c>0$, and for all $(\omega',v')$ sufficiently near $(\omega,v)$ in the $L_2$-norm, i.e., $\exists\,\varepsilon \hsm >\hsm 0$ such that $\|(\omega',v')-(\omega,v)\|_2^2 <\varepsilon $, we have

\vspace{-.4cm}
\begin{equation}
\int_0^{T^*}\hsm\langle\zeta_1, \omega'-\omega\rangle dt  + \int_0^{T^*}\hsm\langle \zeta_2, v'-v)\rangle dt \leq \varphi^k(\omega',v') - \varphi^k(\omega,v) + \bar c\|\omega-\omega'\|_{2}^2 + \bar c \|v-v'\|_{2}^2.\nonumber \end{equation}

\vspace{-.2cm}
Here, $\|\cdot\|_2$ represents the $L_2$-norm. Equivalently,

\vspace{-.7cm}
\begin{eqnarray*}
&& \varphi^k(\omega',v') - \int_0^{T^*}\hsm\langle(\zeta_1(t), \zeta_2(t)),(\omega'(t),v'(t))\rangle dt + \bar c\|\omega-\omega'\|_{2}^2 + \bar c \|v-v'\|_{2}^2 \\
&&  \hspace{6cm}\geq  \int_{0}^{T^*}\hsm |u_L(t)|^2 w(t)dt - \int_0^{T^*}\hsm\langle (\zeta_1(t),\zeta_2(t)),(\omega(t), v(t))\rangle  dt.
\end{eqnarray*}

\vspace{-.4cm}
Denote by $(x',u_L')$ an admissible solution to the approximate lower level problem corresponding to $(\omega',v')$ (or equivalently $(T',y')$). Then,

\vspace{-.6cm}
\begin{eqnarray*}
&&	\int_{0}^{T^*} \hsm |u_L'(t)|^2\omega'(t) dt - \int_0^{T^*}\hsm\langle(\zeta_1(t), \zeta_2(t)),(\omega'(t),v'(t)) \rangle dt + \bar c\|\omega-\omega'\|_{2}^2  + \bar c \|v-v'\|_{2}^2 \\
&&  \hspace{6cm} \geq  \int_{0}^{T^*} \hsm |u_L(t)|^2 \omega(t)dt- \int_0^{T^*}\hsm \langle(\zeta_1(t),\zeta_2(t)),(\omega(t),v(t)) \rangle dt.
\end{eqnarray*}

\vspace{-.2cm}
This means that $(x,u,u_0;\omega,v)$ is an optimal solution to

\vspace{-.5cm}
\begin{eqnarray*}(P_{aux})\mbox{ Minimize }&&  \int_{0}^{T^*}\hsm\hsm\left[|u_L'(t)|^2\omega'(t) - \langle (\zeta_1(t),\zeta_2(t)),(\omega'(t),v'(t))\rangle \right] dt+ \bar c\|\omega-\omega'\|_2^2 + \bar c \|v-v'\|_{2}^2, \\
\mbox{subject to } && \dot x'= \left(\hsmm f(x',u') -u_0' \min \left\{\frac{M}{r_1}, \gamma_k e^{\gamma_k h_L(x',y')}\right\}(x'-y')\hsmm\right)\omega',\;\; \dot y'=v'\omega', \; \;\dot t = \omega'\quad [0,T^*]-\mbox{a.e.}, \\
&& x'(0) \in Q_1 + y_0', \;\; y'(0)=y_0', \;\; y'(T^*) \in \bar  E, \;\; t(0)=0, \\
&& v' \in \mathcal{V}, \;\; u'\in \mathcal{U}, \;\: u_0'\in \mathcal{U}_0, \;\; \omega' \in L_2([0,T^*],\R^+),\mbox{ and }  h_L(x',y') \le 0, \; h_H(y') \leq 0\;\; \mbox{on }[0,T^*].
\end{eqnarray*}

\vspace{-.2cm}
By considering $\overline H_L^k$ is as defined in (\ref{hamilton pontryagin function}), the Hamilton-Pontryagin function for $(P_{aux})$ in the Gamkrelidze's form is

\vspace{-.5cm}
$$H_L^k(y', x', v',u',\bar p,\bar \mu, \bar \lambda,\omega'; v,u,w) :=\left(\overline H_L^k+\bar\lambda \zeta_1\right)\omega'\hsmm +
\hsmm\bar\lambda\langle\zeta_2,v\rangle\hsmm -\hsmm \bar \lambda c\left( |\omega'\hsmm-\hsmm\omega|^2\hsmm+\hsmm |u'\hsmm-\hsmm u|^2\hsmm +\hsmm |v'\hsmm -\hsmm v|^2\right)\hsmm . $$

\vspace{-.2cm}
The application of the necessary conditions of optimality to the reference control process to $ (P_{aux})$ in the Gamkrelidze's form of \cite{arutyunov2011maximum} with nonsmooth data, \cite{KLP-paperIEEE-LCSS2020}, asserts the existence of a multiplier comprising an adjoint arc $\bar p\hsmm =\hsmm (p_H,p_L)\hsmm\in\hsmm AC([0,T^*];\R^{2n})$, nonincreasing scalar function $\bar \mu\hsmm =\hsmm (\mu_H,\mu_L) \hsmm\in\hsmm N\hspace{-.03cm}BV([0,T^*];\R^2)$, with $  \mu_H $ constant on $ \{t\hsmm\in\hsmm[0,T^*]\hsm:\hsmm |y-z|\hsmm >\hsmm R_1\; \forall\, z\hsmm\in \hsmm\partial Q\}$, and $ \mu_L $ constant on $\{t\hsmm\in\hsmm[0,T^*]\hsm:\hsmm |y-x|\hsmm <\hsmm R_1\}$, and a scalar $\bar \lambda\hsmm\geq \hsmm 0$, not all zero, satisfying:
\begin{itemize}
\item[i)] Adjoint and primal system: $\dis (- \dot{\bar p}, \dot y,\dot x) \in \partial^C_{(y',x',\bar p)} H_L^k(y, x, v,u,\bar p,\bar \mu, \bar \lambda,\omega) \quad [0, T^*]-\mbox{a.e.}$,\vspace{.1cm}
\item[ii)] Transversality conditions:  $ p_H(0)\hsm \in\hsm\R^n$, $p_H(T)\hsm\in\hsm - N_{\bar E}(y(T))\hsmm +\hsmm \mu_H(T) (y(T)\hsmm -\hsmm q_0)\hsmm - \hsmm\mu_L (T) (x(T)\hsmm -\hsmm y(T))$, $ p_L(0)\hsm \in\hsm N_{Q_1}(x(0)\hsmm -\hsmm y_0) $, and $p_L(T))\hsmm =\hsmm (\mu_L(0) (x(0)\hsmm - \hsmm y_0),\mu_L(T) (x(T)\hsmm - \hsmm y(T)))$, and \vspace{.1cm}
\item[iii)] Maximum condition on $(\omega',v')$: $\dis \max_{\omega',v'} \left\{\overline H_L^k(v')\omega'\hsmm +\hsmm \bar \lambda  \zeta_1\omega'\hsmm +\hsmm\bar \lambda \langle \zeta_2,v' \rangle\hsmm - \hsmm\bar \lambda \bar c |\omega\hsmm -\hsmm \omega'|^2\hsm -\hsmm \bar \lambda \bar c |v\hsmm -\hsmm v'|^2\right\}\hsmm  =\hsmm (\overline H_L^k (v)\hsmm + \hsmm\bar\lambda \zeta_1)\omega  \hsmm +\hsmm \bar\lambda \langle \zeta_2,v\rangle. $
\end{itemize}
By applying the Lagrange multiplier rule, the maximum condition with respect to $\omega'$, and $v'$, yields

\vspace{-.4cm}
\begin{equation} -\bar \lambda (\zeta_1,\zeta_2) \in \{\overline H_L^k(v)\} \times\nabla_{v} \overline H_L^k(v)\omega +  \{0\} \times N_{\mathcal V}(v). \label{inclusion hamiltonian-value function} \end{equation}

\vspace{-.2cm}
Clearly, this inclusion corresponds to the last inclusion specifying the generalized gradient of $\varphi^k$ of Proposition \ref{subgrad-value}.
		
These conclusions are extended to pairs $(\omega,v)$ for which $\partial^P \hsmm\varphi^k(\omega,v)$ might be an empty set. Take any $(\zeta_1,\zeta_2)\hsmm \in\hsmm \partial \varphi^k(\omega,v)$, where $\dot t\hsm =\hsm \omega$ and $\dot y\hsm =\hsm v$ a.e., such that $v\hsm \in\hsm \mathcal{V}$, and $\omega\hsm \in\hsm  L_2([0,T^*],\R^+)$, and satisfying the corresponding endpoint and state constraints (i.e., $y(0)\hsm =\hsm y_0$, $y(T^*)\hsm\in\hsm \bar  E$, $ t(0)\hsm =\hsm 0 $, and $h_H(y)\hsm\leq\hsm 0$). From the definition of limiting subdifferential, there exists a sequence $\{(\omega_i,v_i)\}$ converging in $L_2$ to $(\omega,v)$ with $\varphi^k(\omega_i,v_i)$ converging to $\varphi^k(\omega,v)$ by continuity of $\varphi^k$, and $\{(\zeta_1^{i},\zeta_2^{i})\}$ satisfying $(\zeta_1^{i},\zeta_2^{i})\hsmm \in \hsmm\partial^P\hsmm \varphi^k(\omega_i,v_i)$, such that, for a fixed $k$, and, for all $i$, $(\zeta_1^{i},\zeta_2^{i})$ converges  weakly in $L_2$ to $(\zeta_1,\zeta_2)$.
		
By compactness arguments, the sequence $\{(\omega_i,v_i)\}$ with $(\omega_i,v_i)\hsmm\in\hsmm  L_2([0,T^*],\R^+)  \hsmm\times\hsmm \mathcal{V}$ can be chosen
so that $\dot t_i\hsmm =\hsmm \omega_i$, and $\dot y_i\hsmm =\hsmm v_i$ a.e., and the corresponding endpoint and state constraints are satisfied. The analysis above implies that, for each one of such $(\omega_i,v_i)$'s, there exists $(p_H^{i},p_L^{i},\mu_H^{i},\mu_L^{i},\bar\lambda^{i})$ associated with the pair $(x_i,u_i)$ such that conditions (i)-(iii) above hold, and that $(x_i,u_i)$ is an optimal solution to $(P_L^k)$, i.e., $(x_i,u_i) \hsmm\in\hsmm\Psi^k(\omega_i,v_i)$.
		
The sequences $\{p_H^{i}\}, \{p_L^i \}$ are bounded, hence $\{\dot p_H^{i}\}, \{\dot p_L^i\}$ are uniformly integrably bounded (as follows from the adjoint system). Then, there exist subsequences (we do not relabel) $\{p_H^{i}\}, \{p_L^i\}$, uniformly converging to some absolutely continuous arcs respectively, $p_H$, and $p_L$. Moreover, we can extract (without relabeling) subsequences $\{\bar \lambda^{i}\}$, $\{\mu_H^{i}\}$, and $\{\mu_L^{i}\}$ converging, respectively, to some $\bar \lambda \hsmm\ge\hsmm 0$, and to some nonincreasing functions of bounded variation $\mu_H$, and $\mu_L$ such that  $\mu_H$ is constant on $\{t\hsmm\in\hsmm[0,T^*]\hsmm:\hsmm |y- z|\hsmm >\hsmm R_1\, \forall\, z\hsmm\in \hsmm\partial Q\}$, and $\mu_L $ is constant on $\{t\hsm\in\hsm[0,T^*]\hsmm: |y- x|\hsmm <\hsmm R_1\}$.
		
Therefore, from the compactness of the feasible set of $(P_L^k)$, we have that $x_i\to x$ uniformly, and $u_i\to u$ a.e., where $(x,u_L)$ is a feasible process to $(P_L^k)$ associated with $(T,y)$ (or equivalently $(\omega,v)$). Moreover, $(x,u_L)\hsmm\in\hsmm \Psi^k(\omega,v)$, since the graph of $\Psi^k(\cdot ,\cdot)$ is closed (as a result of the closure of the graph of $\mathcal{F}^k(T,y)$, and $(x_i,u_i) \hsmm\in\hsmm \Psi^k(\omega_i,v_i)$). Then, by passing to the limit $i\hsmm\to\hsmm \infty$ in the necessary conditions of optimality and by compactness arguments, the statement of Proposition \ref{subgrad-value} regarding the expression of $\partial \varphi^k(\omega,v)$ holds.

Now, it remains to show that $\varphi^k$ is Lipschitz continuous in its domain. By \cite[Theorem 3.6]{CSW1993subgradientcriteria}, $\varphi^k$ is Lipschitz continuous near $(w,v)$ with rank ${\bf L}$ if all the proximal subgradients in $\partial^P\varphi^k(w',v')$ for all $ (w',v')$ of the domain of $\varphi^k$ near $(w,v)$ are bounded in $L_2$ by the constant ${\bf L}$. This can be easily concluded from the necessary conditions obtained above, and the boundedness of the adjoint variable with simple arguments.

\medskip
		
{\bf Proof of Proposition \ref{propo: ekeland principle}.} Let $\bar \theta= (v,u,u_0,\omega)\in\Theta := {\cal V}\times {\cal U}\times{\cal U}_0\times L_2([0,T^*];\R^+)$.
Under assumptions H1-H6, Proposition \ref{prop: existence of solution bilevel} entails the existence of a solution to $(P_F)$ denoted by $( y^*, x^*, z^*, t^*, v^*, u^*, u_0^*, \omega^*)$.
			
Proposition \ref{prop: uniqueness solution} applied to $(D^k)$ cast in the new time parametrization (i.e., on the fixed time interval $[0,T^*]$), yields the existence of a sequence $\{(x_k(0),\bar \theta_k)\}$, with $(x_k,u_k,u_{0,k},\omega_k)\hsmm\in\hsmm {\cal F}_L^k(t_k(T^*),y_k)$, such that $\bar\theta_k \to \bar\theta^*$, a.e. on $[0, T^*]$, and $\{x_k\}$ is such that $x_k \to x^* $ uniformly, being $x^*$ the unique solution to $(P_L)$ with data $ (x^*(0),y^*,t^*,z^*,\bar\theta^*)$.
			
By observing that the construction of the approximating problem $(\widetilde P_F^k)$ involves only the approximation to $(P_L) $ (i.e., $(P_H)$ remains intact), it is clear that Proposition \ref{prop: uniqueness solution} can be readily applied to $(P_F)$ in the reparameterized time by choosing a control process so that $ z^k(T^*)\hsmm - \hsmm \varphi^{k}(\omega_k,v_k)\leq 0$. This is always possible due to the fact that $ {\mathcal F}_L^k(t_k(T^*),y_k)$, with, $\dis t_k(T^*)= \int_0^{T^*}\hsm\hsm \omega_k(s)ds$, is compact and, thus, the existence of solution to $(P_L^k)$ is guaranteed.
			
Since, it is straightforward to conclude that the cost functional of $(P_F^k)$ is lower semi-continuous, and that $\R^n\hsmm\times\hsmm\Theta $ with $|\cdot|\hsmm +\hsmm \Delta(\cdot,\cdot)$ is a complete metric space, we just need to show that there is a sequence of positive numbers $\{\varepsilon_k\}$ satisfying $\varepsilon_k \to 0 $ as $k \to\infty $ so that $(y^*,x^*,z^*,t^*,\bar\theta^*)$ is a $\varepsilon_k^2$-minimizer to $(P_F^k)$. Let $ \varepsilon_k^2 := |x^*(0)\hsmm -\hsmm x_k(0)| \hsmm +\hsmm \Delta(\bar \theta^*,\bar \theta_k)$. From Proposition \ref{prop: uniqueness solution}, it is clear that the sequence $\{\varepsilon_k\} $ satisfies the needed requirements. Thus, Proposition \ref{propo: ekeland principle} follows immediately (for $k$ sufficiently large, $\varepsilon_k < 1$), from Ekeland's Variational Principle.
		
\medskip

{\bf Proof of Proposition \ref{Prop-part-cal}.} Let $(T^*,y^*,x^*,u^*,z^*,u^*_0)$ be a minimizer to $(P_F)$. Then, the partial calmness, cf. Definition \ref{def:partial calmness}, implies that $\exists \rho\hsmm \geq\hsmm 0$, such that, for any feasible control process $(y,x,z,t,\bar\theta)$ of $(P_F)$, we have

\vspace{-.4cm}
\begin{equation} \label{partial calmness} J_H(t^*(T^*),y^*;x^*(0),u_L^*)\leq  J_H(t(T^*),y;x(0),u_L)\hsmm +\hsmm \rho (z(t(T^*))\hsmm -\hsmm \varphi(\omega,v)). \end{equation}

\vspace{-.2cm}
Consider now $(\widetilde P_F^k)$, with $\chi_k=(y_k,x_k,z_k,t_k)$, and $\bar \theta_k = (v_k,u_k,u_{0,k},\omega_k)$ as its state and control variables, respectively. We preserve the notation in the limit by dropping the subscript $k$ and refer to the solution by adding superscript ``*'' to either these variables or their components. Let us assume that $(\widetilde P_F^k)$ is not partially calm. Then, for its minimizer (see Proposition \ref{propo: ekeland principle}), $(\chi^*_k,\bar \theta^*_k)$, there exists a nonnegative sequence $\{\beta_k\}$ converging to some $\beta \hsmm\ge\hsmm 0$, and a control process $(\chi_k,\bar\theta_k)$ feasible to $(\widetilde P_F^{k})$, such that $|x^*(0)\hsmm - \hsmm x_k(0)|\hsmm +\hsmm \Delta (\bar\theta^*,\bar\theta_k) < \varepsilon_k $, and satisfying

\vspace{-.5cm}
\begin{equation}
t^*_k(T^*) > t_k(T^*) + \varepsilon_k\left( |x_k^*(0) \hsmm -\hsmm x_k(0)| \hsmm + \hsmm \Delta (\bar\theta_k^*,\bar\theta_k)\right) \hsmm + \hsmm \beta_k\left(z_k(t_k(T^*))\hsmm -\hsmm \varphi(\omega_k,v_k)\right).\label{partial calmness ineq}
\end{equation}

\vspace{-.2cm}
Moreover, from Proposition \ref{propo: ekeland principle}, we have that the sequence of approximating control processes $\{(\chi^*_k,\bar\theta_k^*)\}$ satisfies

\vspace{-.7cm}
\begin{eqnarray}&& \chi_k^* \to \chi^*  \textrm{ uniformly, and }\vsm\label{convergence trajectories} \\
&& \bar \theta_k^*  \to \bar\theta^* \textrm{ a.e.}\vsm\label{convergence control}
\end{eqnarray}

\vspace{-.3cm}
Since, Proposition \ref{propo: ekeland principle} holds for all $(\chi_k,\bar\theta_k)$ feasible for $(\widetilde P_F^k)$ such that $ |x^*(0)\hsmm - \hsmm x_k(0)| + \Delta(\bar \theta^*,\bar\theta_k) < \varepsilon_k, $
we have that (\ref{convergence control}) asserted for $\{\bar\theta_k^*\}$ also holds for $\{\bar \theta_k\}$. Furthermore, since $ x_k(0)\to x^*(0)$, and,
in the light of the convergence $[0,T^*]$-a.e. of $\{\bar \theta_k\}$ to $\bar \theta^*$, straightforward arguments lead us to conclude that (\ref{convergence trajectories}) also holds for  $\{\chi_k\}$.

Now, since $|x_k^*(0)- x_k(0)|\hsmm\leq\hsmm |x_k^*(0)- x^*(0)|+|x^*(0)-x_k(0)|$, $ \Delta(\bar \theta^*_k,\bar\theta_k)< \Delta(\bar\theta_k^*,\bar\theta^*) \hsmm + \hsmm \Delta(\bar \theta^*,\bar\theta_k)$, and that, by definition $ z_k(t_k(T^*))\hsmm \geq\hsmm \varphi(\omega_k,v_k)$, we conclude, by passing to the limit, while choosing $\beta \hsmm \geq \hsmm \rho$ without any loss of generality, that $ \dis \lim_{k\to\infty} t_k^*(T^*)\hsmm > \hsmm \lim_{k\to\infty} t_k(T^*)$. This is a contradiction since $ t_k^*(T^*)\hsmm \leq \hsmm  t_k(T^*)$ for all $ k$. Thus, $(\widetilde P_F^k)$ has to be partially calm.


\end{document}